\documentclass[10pt]{amsart}
\usepackage{amsmath}
\usepackage[usenames,dvipsnames]{color}
\usepackage{parskip}
\usepackage{amsfonts}
\usepackage{amscd}
\usepackage[centertags]{amsmath}
\usepackage{amssymb}
\usepackage{stmaryrd}
\usepackage[all,cmtip]{xy}
\usepackage[english]{babel}
\usepackage{tabularx}
\usepackage{amsxtra}
\usepackage{euscript}
\usepackage[T1]{fontenc}
\usepackage{doc, exscale, fontenc, latexsym, syntonly}
\usepackage{amsfonts}
\usepackage{amsthm}
\usepackage{graphicx}
\usepackage{mciteplus}
\usepackage{cite}

\newtheorem{theorem}{Theorem}[section]

\newtheorem{lemma}[theorem]{Lemma}
\newtheorem{proposition}[theorem]{Proposition}

\theoremstyle{definition}
\newtheorem{definition}[theorem]{Definition}
\newtheorem{example}[theorem]{Example}
\theoremstyle{remark}

\numberwithin{figure}{section}
\numberwithin{table}{section}

\newcommand*\acknowledgment[1]{%
	\begingroup\noindent
	\rightskip\leftskip
	\begin{flushleft}\textbf{\large Acknowledgment.}\, #1%
		\par\vspace*{1mm}\end{flushleft}\endgroup}
\begin{document}

\title[SOME PROPERTIES OF PROXIMAL HOMOTOPY THEORY]{SOME PROPERTIES OF PROXIMAL HOMOTOPY THEORY}

\author{MEL\.{I}H \.{I}S and \.{I}SMET KARACA}
\date{\today}

\address{\textsc{Melih Is,}
Ege University\\
Faculty of Sciences\\
Department of Mathematics\\
Izmir, Türkiye}
\email{melih.is@ege.edu.tr}
\address{\textsc{Ismet Karaca$^1$,}
Ege University\\
Faculty of Science\\
Department of Mathematics\\
Izmir, Türkiye}
\address{\textsc{Ismet Karaca$^2$,}
	Azerbaijan State Agrarian University\\
	Faculty Agricultural Economics\\
	Department of Agrarian Economics\\
	Gence, Azerbaijan}
\email{ismet.karaca@ege.edu.tr}



\begin{abstract}
    Nearness theory comes into play in homotopy theory because the notion of closeness between points is essential in determining whether two spaces are homotopy equivalent. While nearness theory and homotopy theory have different focuses and tools, they are intimately connected through the concept of a metric space and the notion of proximity between points, which plays a central role in both areas of mathematics. This manuscript investigates some concepts of homotopy theory in proximity spaces. Moreover, these concepts are taken into account in descriptive proximity spaces.
\end{abstract}

\maketitle

\section{Introduction}
\label{intro}

\quad Topological perspective first appears in the scientific works of Riemann and Poincare in the 19th century\cite{Riemann:1851,Poincare:1895}. The concept reveals that the definitions of topological space emerge either through Kuratowski's closure operator\cite{Kuratowski:1958} or through the use of  open sets. Given the Kuratowski's closure operator, there are many strategies and approaches that seem useful in different situations and are worth developing, as in nearness theory. Proximity spaces are created by reflection of the concept of being near/far on sets. Given a nonempty set $X$, and any subsets $E$, $F \subset X$, we say that $E$ is near $F$ if $E \cap F \neq \emptyset$. A method based on the idea of near sets is first proposed by Riesz, is revived by Wallace, and is axiomatically elaborated it by Efremovic\cite{Riesz:1908,Wallace:1941,Efremovic:1952}. Let $X$ be a nonempty set. A proximity is a binary relation (actually a nearness relation) defined on subsets of $X$ and generally denoted by $\delta$. One can construct a topology on $X$ induced by the pair $(X,\delta)$ using the closure operator (named a proximity space). Indeed, for any point $x \in X$, if $\{x\}$ is near $E$, then $x \in \overline{E}$. In symbols, if $\{x\} \delta E$, then $x \delta \overline{E}$. It should be noted that the notation $x \delta E$ is sometimes used instead of $\{x\} \delta E$ for abbreviation (in particular in a metric space). It appears that several proximities may correspond in this way to the same topology on $X$. Moreover, several topological conclusions can be inferred from claims made about proximity spaces.

\quad The near set theory is reasonably improved by Smirnov's compactification, Leader's non-symmetric proximity, and Lodato's symmetric proximity\cite{Smirnov:1952,Leader:1964,Lodato:1964}. 
Peters also contributes to the theory of nearness by introducing the concept of spatial nearness and descriptive nearness\cite{Peters1:2007,Peters2:2007}. In addition, the strong structure of proximity spaces stands out in the variety of application areas: In\cite{NaimpallyPeters:2013}, it is possible to see the construction of proximity spaces in numerous areas such as cell biology, the topology of digital images, visual marketing, and so on. In a broader context, the application areas of near spaces are listed along with the history of the subject in \cite{PetersNaimpally:2012}. According to this, some near set theory-related topics are certain engineering problems, image analysis, and human perception. The main subject of this article, algebraic topology approaches in proximity spaces, is a work in progress in the literature. Mapping spaces, one of the fundamental concepts in homotopy theory, is examined in proximity spaces in \cite{PeiRen:1985}. The proximal setting of the notion fibration is first defined in \cite{PetersTane2:2022}. Peters and Vergili have recently published interesting research on descriptive proximal homotopy, homotopy cycles, path cycles, and Lusternik-Schnirelmann theory of proximity spaces\cite{PetersTane:2021,PetersTane2:2021,PetersTane:2022,PetersTane2:2022}.  

\quad This paper is primarily concerned with the theory of proximal homotopy and is organized as follows. In Section \ref{sec:1}, we discuss the general properties of proximity and descriptive proximity spaces. Section \ref{sec:2} covered 4 main topics in proximity spaces: Mapping spaces, covering spaces, fibrations, and cofibrations. They provide different types of examples and frequently used algebraic topology results in proximal homotopy cases. Next, the descriptive proximal homotopy theory, which is handled in Section \ref{sec:4}, discusses the ideas from Section \ref{fig:3} and illustrates them with examples by using feature vectors as color scales. Finally, the last section establishes the direction for future works by clearly emphasizing the application areas of homotopy theory.

\section{Preliminaries}
\label{sec:1}

\quad Before proceeding with the main results, it is critical to remember the fundamental characteristics of proximity and descriptive proximity spaces.

\subsection{On Proximity Spaces}
\label{sub1}

Consider a pseudo-metric space $(X,d)$. A binary relation $\delta$ defined by
\begin{eqnarray*}
	''E \delta F \ \ \Leftrightarrow \ \ D(E,F) = 0''
\end{eqnarray*}
satisfies 
\begin{eqnarray*}
	&& \textbf{(a)} \hspace*{1.0cm} E \delta F \ \ \Rightarrow \ \ F \delta E,\\
	&& \textbf{(b)} \hspace*{1.0cm} (E \cup F) \delta G \ \ \Leftrightarrow \ \ E \delta G \ \vee \ F \delta G,\\
	&& \textbf{(c)} \hspace*{1.0cm} E \delta F \ \ \Rightarrow \ \ E \neq \emptyset \ \land \ F \neq \emptyset,\\
	&& \textbf{(d)} \hspace*{1.0cm} E \underline{\delta} F \ \ \Rightarrow \ \ \exists G \subset X: E \underline{\delta} G \ \land \ (X-G) \underline{\delta} F,\\
	&& \textbf{(e)} \hspace*{1.0cm} E \cap F \neq \emptyset \ \ \Rightarrow \ \ E \delta F
\end{eqnarray*}
for $D(E,F) = \inf\{d(x_{1},x_{2}): x_{1} \in E, x_{2} \in F\}$\cite{Efremovic:1952,Smirnov:1952,NaimpallyWarrack:1970}.

\quad $\delta$ is a nearness relation and $E \delta F$ is read as “$E$ \textit{is near} $F$''. Otherwise, the notation $E \underline{\delta} F$ means that “$E$ is \textit{far from} $F$''.

\begin{definition}\cite{Efremovic:1952,Smirnov:1952,NaimpallyWarrack:1970}
	The nearness relation $\delta$ for the subsets of $X$ is said to be an \textit{Efremovic proximity} (simply denoted by \textit{EF-proximity} or \textit{proximity}) provided that $\delta$ satisfies \textbf{(a)}-\textbf{(e)}. $(X,\delta)$ is said to be an \textit{EF-proximity} (or \textit{proximity}) \textit{space}.
\end{definition}

\quad As an example of a proximity space, the \textit{discrete proximity} $\delta$ on a (nonempty) set $X$ is defined by “$E \delta F \ \Leftrightarrow \ E \cap F \neq \emptyset$'' for $E$, $F \subset X$. Also, the \textit{indiscrete proximity} $\delta^{'}$ on a (nonempty) set $X$ is given by $E \delta^{'} F$ for any nonempty subsets $E$ and $F$ in $X$. A subset $E$ of $X$ with a proximity $\delta$ is a closed set if “$x \delta E \ \Rightarrow \ x \in E$''. The converse is also valid. Therefore, given a proximity $\delta$ on $X$, a topology $\tau(\delta)$ is defined by the family of complements of all closed sets via Kuratowski closure operator\cite{NaimpallyWarrack:1970}.   

\begin{theorem}\cite{NaimpallyWarrack:1970}
	For a proximity $\delta$ and a topology $\tau(\delta)$ on a set $X$, we have that the closure $\overline{E}$ coincides with $\{x : x \delta E\}$. 
\end{theorem}

\quad Given any proximities $\delta$ and $\delta^{'}$ on respective sets $X$ and $X^{'}$, a map $h$ from $X$ to $X^{'}$ is called \textit{proximally continuous} if “$E \delta F \ \Rightarrow \ h(E) \delta^{'} h(F)$'' for $E$, $F \subset X$\cite{Efremovic:1952,Smirnov:1952}. We denote a proximally continuous map by “pc-map''. Given a proximity $\delta$ on $X$ and a subset $E \subset X$, a \textit{subspace proximity} $\delta_{E}$ is defined on the subsets of $E$ as follows\cite{NaimpallyWarrack:1970}: “$E_{1} \delta E_{2} \ \Leftrightarrow \ E_{1} \delta_{E} E_{2}$'' for $E_{1}$, $E_{2} \subset E$. Let $(X,\delta)$ be a proximity space and $(E,\delta_{E})$ a subspace proximity. A pc-map $k : (X,\delta) \rightarrow (E,\delta_{E})$ is a \textit{proximal retraction} provided that $k \circ j$  is an identity map on $1_{E}$, where $j : (E,\delta_{E}) \rightarrow (X,\delta)$ is an inclusion map.

\begin{lemma}\label{l1}\cite{PetersTane:2021}(Gluing Lemma)
	Assume that $f_{1} : (X^{'},\delta_{1}^{'}) \rightarrow (Y^{'},\delta_{2}^{'})$ and \linebreak $f_{2} : (X^{''},\delta_{1}^{''}) \rightarrow (Y^{'},\delta_{2}^{'})$ are pc-maps with the property that they agree on the intersection of $X$ and $X^{''}$. Then the map $f_{1} \cup f_{2} : (X^{'} \cup X^{''},\delta) \rightarrow (Y^{'},\delta_{2}^{'})$, defined by $f_{1} \cup f_{2}(s) = \begin{cases}
		f_{1}(s), & s \in X^{'} \\
		f_{2}(s), & s \in X^{''}
	\end{cases}$ for any $s \in X^{'} \cup X^{''}$, is a pc-map.
\end{lemma}

\quad We say that $h$ is a \textit{proximity isomorphism} provided that $h$ is a bijection and each of $h$ and $h^{-1}$ is pc-map\cite{NaimpallyWarrack:1970}. According to this, $(X,\delta)$ and $(X^{'},\delta^{'})$ are said to be \textit{proximally isomorphic spaces}. Another important proximity relation is given on the subsets of the cartesian product of two proximity spaces as follows\cite{Leader:1964}: Let $\delta$ and $\delta^{'}$ be any proximities on respective sets $X$ and $X^{'}$. For any subsets $E_{1} \times E_{2}$ and $F_{1} \times F_{2}$ of $X \times X^{'}$, $E_{1} \times E_{2}$ \textit{is near} $F_{1} \times F_{2}$ if $E_{1} \delta F_{1}$ and $E_{2} \delta^{'} F_{2}$.

\begin{definition}\label{d8}\cite{PetersTane:2021}
	Given two pc-maps $h_{1}$ and $h_{2}$ from $X$ to $X^{'}$, if there is a pc-map $F$ from $X \times I$ to $X^{'}$ with the properties $F(x,0) = h_{1}(x)$ and $F(x,1) = h_{2}(x)$, then $h_{1}$ and $h_{2}$ are called \textit{proximally homotopic maps}. 
\end{definition}

\quad The map $F$ in Definition \ref{d8} is said to be a \textit{proximal homotopy between $h$ and $h^{'}$}. We simply denote a proximal homotopy by “prox-hom''. Similar to topological spaces, prox-hom is an equivalence relation on proximity spaces. Let $\delta$ be a proximity on $X$ and $E \subset X$. $E$ is called a \textit{$\delta-$neighborhood of $F$}, denoted by $F \ll_{\delta} E$, provided that $F \underline{\delta} (X-E)$\cite{NaimpallyWarrack:1970}. The proximal continuity of any function $h : (X,\delta) \rightarrow (X^{'},\delta^{'})$ can also be expressed as 
\begin{eqnarray*}
	''E \ll_{\delta^{'}} F \ \Rightarrow \ h^{-1}(E) \ll_{\delta} h^{-1}(F)''
\end{eqnarray*} 
for any $E$, $F \subset X'$.

\begin{theorem}\cite{NaimpallyWarrack:1970}
	Let $E_{k} \ll_{\delta} F_{k}$ for $k = 1,\cdots,r$. Then
	\begin{eqnarray*}
		\displaystyle\bigcap_{k=1}^{r}E_{k} \ll_{\delta} \displaystyle\bigcap_{k=1}^{r}F_{k} \ \ \ \text{and} \ \ \ \displaystyle\bigcup_{k=1}^{r}E_{k} \ll_{\delta} \displaystyle\bigcup_{k=1}^{r}F_{k}.
	\end{eqnarray*}
\end{theorem}

\begin{definition}\label{d7}\cite{PetersTane:2021}
	For any two elements $x_{1}$ and $x_{2}$ in $X$ with a proximity $\delta$, a \textit{proximal path from $x_{1}$ to $x_{2}$} in $X$ is a pc-map $h$ from $I = [0,1]$ to $X$ for which $h(0) = x_{1}$ and $h(1) = x_{2}$.
\end{definition}

\quad The proximal continuity of the proximal path $h : I \rightarrow X$ in Definition \ref{d7} means that “$D(E,F) = 0 \ \Rightarrow \ h(E) \delta h(F)$'' for $E$, $F \in I$. Recall that $X$ is a \textit{connected proximity space} if and only if for all nonempty $E$, $F \in \mathcal{P}(X)$, $E \cup F = X$ implies that $E \delta F$\cite{MrowkaPervin:1964}. Let $\delta$ be a proximity on $X$. Then $X$ is called a \textit{path-connected proximity space} if, for any points $x_{1}$ and $x_{2}$ in $X$, there exists a proximal path from $x_{1}$ to $x_{2}$ in $X$.

\begin{lemma}\label{l2}
	Proximal path-connectedness implies proximal connectedness as in the same as topological spaces. 
\end{lemma}

\begin{proof}
	Let $\delta$ be a path-connected proximity on $X$. Suppose that $(X,\delta)$ is not proximally connected. Then there exists two nonempty subsets $E$, $F$ in $X$ such that $E \cup F = X$ and $E \underline{\delta} F$. Since $X$ is proximally path-connected, there is a pc-map $h : [0,1] \rightarrow X$ with $h(0) = E$ and $h(1) = F$. Consider the subsets $h^{-1}(E)$ and $h^{-1}(F) \in I$. They are nonempty sets because \linebreak $0 \in h^{-1}(E)$ and $1 \in h^{-1}(F)$. Their union is $[0,1]$, and by the proximal continuity of $h$, $h^{-1}(E) \underline{\delta} h^{-1}(F)$. This contradicts with the fact that $[0,1]$ is proximally connected. Finally, $X$ is proximally connected.
\end{proof}

\begin{theorem}\label{t2}
	Proximal path-connectedness coincides with proximal connectedness.
\end{theorem}

\begin{proof}
	Given a proximity $\delta$ on $X$, by Lemma \ref{l2}, it is enough to prove that any connected proximity space is a path-connected proximity space. Suppose that $X$ is not a path-connected proximity space. Then any map $h : ([0,1],\delta^{'}) \rightarrow (X,\delta)$ with $h(0) = x$ and $h(1) = y$ is not proximally continuous, i.e., if $E \delta^{'} F$ for all $E$, $F \in [0,1]$, then $h(E) \underline{\delta} h(F)$. Take $E = \{0\} \subset I$ and $F = (0,1] \subset I$. Since $D(E,F) = \inf\{d(0,z) : z \in F\} = 0$, we have that $E \delta F$. It follows that $h(E) = \{x\}$ is not near $h(F) = X \setminus \{x\}$. On the other hand,
	\begin{eqnarray*}
		h(E) \cup h(F) = \{x\} \cup X \setminus \{x\} = X.
	\end{eqnarray*}
	Thus, $X$ is not proximally connected and this is a contradiction.
\end{proof}

\subsection{On Descriptive Proximity Spaces}
\label{sub2}

Assume that $X$ is a nonempty set and $x \in X$. Consider the set $\Phi = \{\phi_{1},\cdots,\phi_{m}\}$ of maps (generally named as probe functions) $\phi_{j} : X \rightarrow \mathbb{R}$, $j = 1,\cdots,m$, such that $\phi_{j}(x)$ denotes a feature value of $x$. Let $E \subset X$. Then the set of descriptions of a point $e$ in $E$, denoted by $\mathcal{Q}(E)$, is given by the set $\{\Phi(e) : e \in E\}$, where $\Phi(e)$ (generally called a feature vector for $e$) equals $(\phi_{1}(e),\cdots,\phi_{m}(e))$. For $E$, $F \subset X$, the binary relation $\delta_{\Phi}$ is defined by 
\begin{eqnarray}
	''E \delta_{\Phi} F \ \Leftrightarrow \ \mathcal{Q}(E) \cap \mathcal{Q}(F) \neq \emptyset'',
\end{eqnarray}
and $E \delta_{\Phi} F$ is read as “$E$ is \textit{descriptively near} $F$''\cite{Peters1:2007,Peters2:2007,Peters:2013}. Also, $E \underline{\delta_{\Phi}} F$ is often used to state “$E$ is \textit{descriptively far from} $F$''. The \textit{descriptive intersection of $E$ and $F$} and the \textit{descriptive union of $E$ and $F$} are defined by
\begin{eqnarray*}
	E \displaystyle \bigcap_{\Phi} F = \{x \in E \cup F : \Phi(x) \in \mathcal{Q}(E) \ \land \ \Phi(x) \in \mathcal{Q}(F)\},
\end{eqnarray*}
and
\begin{eqnarray*}
	E \displaystyle \bigcup_{\Phi} F = \{x \in E \cup F : \Phi(x) \in \mathcal{Q}(E) \ \vee  \ \Phi(x) \in \mathcal{Q}(F)\},
\end{eqnarray*}
respectively\cite{Peters:2013}.

A binary relation $\delta_{\Phi}$ defined by (1) \cite{NaimpallyPeters:2013}
satisfies
\begin{eqnarray*}
	&&\textbf{(f)} \hspace*{1.0cm} E \delta_{\Phi} F \ \ \Rightarrow \ \ E \neq \emptyset \ \land F \neq \emptyset,\\
	&&\textbf{(g)} \hspace*{1.0cm} E \displaystyle \bigcap_{\Phi} F \neq \emptyset \ \ \Rightarrow \ \ E \delta_{\Phi} F,\\
	&&\textbf{(h)} \hspace*{1.0cm} E \displaystyle \bigcap_{\Phi} F \neq \emptyset \ \ \Rightarrow \ \ F \displaystyle \bigcap_{\Phi} E,\\
	&&\textbf{(i)} \hspace*{1.0cm} E \delta_{\Phi} (F \cup G) \ \ \Leftrightarrow \ \ E \delta_{\Phi} F \ \vee \ E \delta_{\Phi} G,\\
	&&\textbf{(k)} \hspace*{1.0cm} E \underline{\delta_{\Phi}} F \ \ \Rightarrow \ \ \exists G \subset X: E \underline{\delta_{\Phi}} G \ \land \ (X-G) \underline{\delta_{\Phi}} F.
\end{eqnarray*}

\quad $\delta_{\Phi}$ is a descriptive nearness relation.

\begin{definition}\cite{NaimpallyPeters:2013}
	The nearness relation $\delta_{\Phi}$ for the subsets of $X$ is said to be an \textit{descriptive Efremovic proximity} (simply denoted by \textit{descriptive EF-proximity} or \textit{descriptive proximity}) if $\delta_{\Phi}$ satisfies \textbf{(f)}-\textbf{(k)}. $(X,\delta_{\Phi})$ is said to be a \textit{descriptive EF-proximity} (or \textit{descriptive proximity}) \textit{space}.
\end{definition}

\quad A map $h : (X,\delta_{\Phi}) \rightarrow (X,\delta_{\Phi}^{'})$ is called \textit{descriptive proximally continuous} provided that “$E \delta_{\Phi} F \ \Rightarrow \ h(E) \delta_{\Phi}^{'} h(F)$'' for $E$, $F \subset X$\cite{Peters:2014,PetersTane:2021}. We denote a descriptive proximally continuous map by “dpc-map''. Let $\delta_{\Phi}$ be a descriptive proximity on $X$, and $E \subset X$ a subset. Then a \textit{descriptive subspace proximity} $\delta_{\Phi}^{E}$ is defined on the subsets of $E$ as follows: 
\begin{eqnarray*}
	''E_{1} \delta_{\Phi} E_{2} \ \Leftrightarrow \ E_{1} \delta_{\Phi}^{E} E_{2}''
\end{eqnarray*}
for $E_{1}$, $E_{2} \subset E$. Given a descriptive proximity $\delta_{\Phi}$ on $X$, a descriptive subspace proximity $(E,\delta_{\Phi}^{E})$, and the inclusion $j : (E,\delta_{\Phi}^{E}) \rightarrow (X,\delta_{\Phi})$, a dpc-map \linebreak$k : (X,\delta_{\Phi}) \rightarrow (E,\delta_{\Phi}^{E})$ is called a \textit{descriptive proximal retraction} if $k \circ j = 1_{E}$.

\begin{lemma}\label{l3}\cite{PetersTane:2021}(Gluing Lemma)
	Assume that $f_{1} : (X^{'},\delta_{\Phi_{1}}^{'}) \rightarrow (Y^{'},\delta_{\Phi_{2}}^{'})$ and $f_{2} : (X^{''},\delta_{\Phi_{1}}^{''}) \rightarrow (Y^{'},\delta_{\Phi_{2}}^{'})$ are two dpc-maps with the property that they agree on the intersection of $X^{'}$ and $X^{''}$. Then the map $f_{1} \cup f_{2}$ from $(X^{'} \cup X^{''},\delta_{\Phi})$ to $(Y^{'},\delta_{\Phi_{2}}^{'})$, defined by $f_{1} \cup f_{2}(s) = \begin{cases}
		f_{1}(s), & s \in X^{'} \\
		f_{2}(s), & s \in X^{''}
	\end{cases}$ for any $s \in X^{'} \cup X^{''}$, is a dpc-map.
\end{lemma}

\quad $h$ is a \textit{descriptive proximity isomorphism} if $h$ is a bijection and each of $h$ and $h^{-1}$ is dpc-map\cite{NaimpallyWarrack:1970}. Hence, $(X,\delta_{\Phi})$ and $(X^{'},\delta_{\Phi}^{'})$ are called \textit{descriptive proximally isomorphic spaces}. A descriptive proximity relation on the cartesian product of descriptive proximity spaces is defined as follows\cite{Leader:1964}: Assume that $\delta_{\Phi}$ and $\delta_{\Phi}^{'}$ are any descriptive proximities on $X$ and $X^{'}$, respectively. $E \delta_{\Phi} F$ and $E^{'} \delta_{\Phi}^{'} F^{'}$ implies that $E \times E^{'}$ is \textit{descriptively near} $F \times F^{'}$, where $E \times E^{'}$ and $F \times F^{'}$ are any subsets of $X \times X^{'}$.

\begin{definition}\label{d9}\cite{PetersTane:2021}
	Let $h_{1}$, $h_{2} : (X,\delta_{\Phi}) \rightarrow (X^{'},\delta_{\Phi}^{'})$ be any map. Then $h_{1}$ and $h_{2}$ are said to be \textit{descriptive proximally homotopic maps} provided that there exists a dpc-map $G : X \times I \rightarrow X^{'}$ with $G(x,0) = h_{1}(x)$ and $G(x,1) = h_{2}(x)$. 
\end{definition}

\quad In Definition \ref{d9}, $G$ is a \textit{descriptive proximal homotopy between $h_{1}$ and $h_{2}$}. We simply denote a descriptive proximal homotopy by “dprox-hom''. Given a descriptive proximity $\delta_{\Phi}$ on $X$ and a subset $F \subset X$, $F$ is said to be a \textit{$\delta_{\Phi}-$neighborhood of $E$}, denoted by $E \ll_{\delta_{\Phi}} F$, if $E \underline{\delta_{\Phi}} (X-F)$\cite{Peters:2014}. 

\begin{theorem}\cite{NaimpallyWarrack:1970}
	Let $E_{j} \ll_{\delta_{\Phi}} F_{j}$ for $j = 1,\cdots,m$. Then
	\begin{eqnarray*}
		\displaystyle\bigcap_{j=1}^{m}E_{j} \ll_{\delta_{\Phi}} \displaystyle\bigcap_{j=1}^{m}F_{j} \ \ \ \text{and} \ \ \ \displaystyle\bigcup_{j=1}^{m}E_{j} \ll_{\delta_{\Phi}} \displaystyle\bigcup_{j=1}^{m}F_{j}.
	\end{eqnarray*}
\end{theorem}

\begin{definition}\label{d10}\cite{PetersTane:2021}
	Let $x_{1}$ and $x_{2}$ be any two elements in $X$ with a descriptive proximity $\delta_{\Phi}$. Then a \textit{descriptive proximal path from $x_{1}$ to $x_{2}$} in $X$ is a dpc-map $h$ from $I = [0,1]$ to $X$ for which $h(0) = x_{1}$ and $h(1) = x_{2}$.
\end{definition}

\quad In Definition \ref{d10}, the fact $h : I \rightarrow X$ is descriptive proximally continuous means that “$D(E,F) = 0 \ \Rightarrow \ h(E) \delta_{\Phi} h(F)$'' for $E$, $F \in I$. A descriptive proximity space $(X,\delta_{\Phi})$ is \textit{connected} if and only if for all nonempty $E$, $F \in \mathcal{P}(X)$, $E \cup F = X$ implies that $E \delta_{\Phi} F$\cite{MrowkaPervin:1964}. A descriptive proximity space $(X,\delta_{\Phi})$ is \textit{path-connected} if, for any points $x_{1}$ and $x_{2}$ in $X$, there exists a descriptive proximal path from $x_{1}$ to $x_{2}$ in $X$.

\begin{theorem}
	In a descriptive proximity space, path-connectedness coincides with connectedness.
\end{theorem}

\begin{proof}
	Follow the method in the proof of Theorem \ref{t2}.
\end{proof}

\section{Homotopy Theory on Proximity Spaces}
\label{sec:2}

\quad This section, one of the main parts (Section \ref{sec:2} and Section \ref{sec:3}) of the paper, examines the projection of the homotopy theory elements in parallel with the proximity spaces. First, we start with the notion of proximal mapping spaces. Then we have proximal covering spaces. The last two parts are related to proximal fibrations and its dual notion of proximal cofibrations. Results on these four topics that we believe will be relevant to future proximity space research are presented.

\subsection{Proximal Mapping Spaces}
\label{subsec:1}

The work of mapping spaces in nearness theory starts with \cite{PeiRen:1985} and is still open to improvement. Note that the study of discrete invariants of function spaces is essentially homotopy theory in algebraic topology, and recall that depending on the nature of the spaces, it may be useful to attempt to impose a topology on the space of continuous functions from one topological space to another. One of the best-known examples of this is the compact-open topology.

\begin{definition}
	Let $\delta_{1}$ and $\delta_{2}$ be two proximities on $X$ and $Y$, respectively. The proximal mapping space $Y^{X}$ is defined as $\{\alpha : X \rightarrow Y \ | \ \alpha \ \text{is a pc-map}\}$ having the following proximity relation $\delta$ on itself: Let $E$, $F \subset X$ and $\{\alpha_{i}\}_{i \in I}$ and $\{\beta_{j}\}_{j \in J}$ be any subsets of pc-maps in $Y^{X}$. We say that $\{\alpha_{i}\}_{i \in I} \delta \{\beta_{j}\}_{j \in J}$ if the fact $E \delta_{1} F$ implies that $\alpha_{i}(E) \delta_{2} \beta_{j}(F)$.
\end{definition}

\begin{figure}[h]
	\centering
	\includegraphics[width=0.30\textwidth]{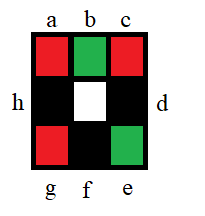}
	\caption{The picture represented by $X = \{a,b,c,d,e,f,g,h\}$.}
	\label{fig:1}
\end{figure}

\begin{example}
	Consider the set $X = \{a,b,c,d,e,f,g,h\}$ of cells in Figure \ref{fig:1} with the proximity $\delta$. Define three proximal paths $\alpha_{1}$, $\alpha_{2}$, and $\alpha_{3} \in X^{I}$ by
	\begin{eqnarray*}
		&&\alpha_{1} : a \mapsto b \mapsto c \mapsto d \mapsto e \mapsto f \mapsto g \mapsto h, \\
		&& \alpha_{2} : h \mapsto a \mapsto b \mapsto c \mapsto d \mapsto e \mapsto f \mapsto g, \\
	    && \alpha_{3} : a \mapsto h \mapsto g \mapsto f \mapsto e \mapsto d \mapsto c \mapsto b.
	\end{eqnarray*}
    For all $t \in I$, $\alpha_{1}(t) \delta \alpha_{2}(t)$. This means that $\alpha_{1}$ is near $\alpha_{2}$. On the other hand, for $t \in [2/8,3/8]$, we have that $\alpha_{1}(t) = c$ and $\alpha_{3}(t) = g$, that is, $\alpha_{1}$ and $\alpha_{3}$ are not near in $X$.
\end{example}

\begin{definition}
	For the proximal continuity of a map $H : (X,\delta_{1}) \rightarrow (Z^{Y},\delta^{'})$, we say that the fact $E \delta_{1} F$ implies that $H(E) \delta^{'} H(F)$ for any subsets $E$, $F \subset X$.
\end{definition} 

\begin{proposition}\label{p1}
	Let $\delta_{1}$, $\delta_{2}$, and $\delta_{3}$ be any proximities on $X$, $Y$, and $Z$, respectively. Then the map $G : (X \times Y,\delta^{''}) \rightarrow (Z,\delta_{3})$ is pc-map if and only if the map \linebreak$H : (X,\delta_{1}) \rightarrow (Z^{Y},\delta^{'})$ defined by $H(E)(F) := G(E \times F)$ is pc-map for $E \subset X$ and $F \subset Y$.
\end{proposition}

\begin{proof}
	Assume that $E_{1} \delta_{1} F_{1}$ for $E_{1}$, $F_{1} \subset X$. If $E_{2} \delta_{2} F_{2}$ for $E_{2}$, $F_{2} \subset Y$, then we find $(E_{1} \times E_{2}) \delta^{''} (F_{1} \times F_{2})$. Since $G$ is a pc-map, we get $G(E_{1} \times E_{2}) \delta_{3} G(F_{1} \times F_{2})$. It follows that $H(E_{1})(E_{2}) \delta_{3} H(F_{1})(F_{2})$. This shows that $H(E_{1}) \delta^{'} H(F_{1})$, i.e., $H$ is a pc-map. Conversely, assume that $(E_{1} \times E_{2}) \delta^{''} (F_{1} \times F_{2})$. Then we get $E_{1} \delta_{1} F_{1}$ in $X$ and $E_{2} \delta_{2} F_{2}$ in $Y$. Since $H$ is a pc-map, we get $H(E_{1}) \delta^{'} H(F_{1})$. So, we have that $H(E_{1})(E_{2}) \delta_{3} H(F_{1})(F_{2})$. This leads to the fact that $G(E_{1} \times E_{2}) \delta^{''} G(F_{1} \times F_{2})$, namely that, $G$ is a pc-map. 
\end{proof}

\begin{theorem}
	Let $\delta_{1}$, $\delta_{2}$, and $\delta_{3}$ be any proximities on $X$, $Y$, and $Z$, respectively. Then $(Z^{X \times Y},\delta_{4})$ and $((Z^{Y})^{X},\delta_{5})$ are proximally isomorphic spaces. 
\end{theorem}

\begin{proof}
	Define a bijective map $f : Z^{X \times Y} \rightarrow (Z^{Y})^{X}$ by $f(G) = H$. For any pc-maps $G$, $G^{'} \subset Z^{X \times Y}$ such that $G \delta_{4} G^{'}$, we have that $f(G) \delta_{5} f(G^{'})$. Indeed, for $E_{1} \times E_{2}$, $F_{1} \times F_{2} \subset X \times Y$, we have that $G(E_{1} \times E_{2}) \delta_{3} G(F_{1} \times F_{2})$. This means that $H(E_{1})(E_{2}) \delta_{3} H(F_{1})(F_{2})$. Another saying, we find $H \delta_{5} H^{'}$. Therefore, $f$ is a pc-map. For the proximal continuity of $f^{-1}$, assume that $H \delta_{5} H^{'}$. Then we have that $H(E_{1})$ and $H^{'}(F_{1})$ are near in $Z^{Y}$ for $E_{1}$, $F_{1} \subset X$. If $E_{2} \delta_{2} F_{2}$ in $Y$, then we have that $H(E_{1})(E_{2}) \delta_{3} H^{'}(F_{1})(F_{2})$. It follows that $G(E_{1} \times E_{2}) \delta_{3} G^{'}(F_{1} \times F_{2})$. Thus, we obtain that $G \delta_{4} G^{'}$, which means that $f^{-1}(H) \delta_{4} f^{-1}(H^{'})$. Finally, $f$ is a proximity isomorphism. 
\end{proof}

\begin{theorem}\label{t1}
	Let $\delta_{1}$, $\delta_{2}$, and $\delta_{3}$ be any proximities on $X$, $Y$, and $Z$, respectively. Then $((Y \times Z)^{X},\delta_{4})$ and $(Y^{X} \times Z^{X},\delta_{5})$ are proximally isomorphic spaces. 
\end{theorem}

\begin{proof}
	The proximal isomorphism is given by the map \[f : ((Y \times Z)^{X},\delta_{4}) \rightarrow (Y^{X} \times Z^{X},\delta_{5})\] 
	with $f(\alpha) = (\pi_{1} \circ \alpha,\pi_{2} \circ \alpha)$, where $\pi_{1}$ and $\pi_{2}$ are the projection maps from $Y \times Z$ to the respective spaces. For any $\{\alpha_{i}\}_{i \in I}$, $\{\beta_{j}\}_{j \in J} \subset (Y \times Z)^{X}$ such that $\{\alpha_{i}\}_{i \in I}$ is near $\{\beta_{j}\}_{j \in J}$, we obtain that $\pi_{k} \circ \{\alpha_{i}\}_{i \in I}$ is near $\pi_{k} \circ \{\beta_{i}\}_{j \in J}$ for each $k \in \{1,2\}$. Therefore, we have that $(\pi_{1} \circ \{\alpha_{i}\}_{i \in I},\pi_{2} \circ \{\alpha_{i}\}_{i \in I})$ is near $(\pi_{1} \circ \{\beta_{j}\}_{j \in J},\pi_{2} \circ \{\beta_{j}\}_{j \in J})$. Thus, $f(\{\alpha_{i}\}_{i \in I})$ is near $f(\{\beta_{j}\}_{j \in J})$, i.e., $f$ is a pc-map. For the pc-map
	\[g : (Y^{X} \times Z^{X},\delta_{5}) \rightarrow ((Y \times Z)^{X},\delta_{4})\]
	with $g(\beta,\gamma) = (\beta \times \gamma) \circ \Delta_{X}$, where $\Delta_{X} : (X,\delta_{1}) \rightarrow (X^{2},\delta_{1}^{'})$ is a diagonal map of proximity spaces on $X$, we have that $g \circ f$ and $f \circ g$ are identity maps on respective proximity spaces $(Y \times Z)^{X}$ and $Y^{X} \times Z^{X}$. Consequently, $((Y \times Z)^{X},\delta_{4})$ and $(Y^{X} \times Z^{X},\delta_{5})$ are proximally isomorphic spaces. 
\end{proof}

\begin{definition}
	Let $\delta_{1}$ and $\delta_{2}$ be any proximities on $X$ and $Y$, respectively. Then the proximal evaluation map \[e_{X,Y} : (Y^{X} \times X,\delta) \rightarrow (Y,\delta_{2})\] is defined by $e(\alpha,x) = \alpha(x)$.
\end{definition}  

\quad To show that the evaluation map $e_{X,Y}$ is  a pc-map, we first assume that $(\{\alpha_{i}\}_{i \in I} \times E) \delta (\{\beta_{j}\}_{j \in J} \times F)$ in $Y^{X} \times X$. This means that $\{\alpha_{i}\}_{i \in I} \delta^{'} \{\beta_{j}\}_{j \in J}$ for a proximity relation $\delta^{'}$ on $Y^{X}$ and $E \delta_{1} F$ in $X$. It follows that $\alpha_{i}(E) \delta_{2} \beta_{j}(F)$ in $Y$ for any $i \in I$ and $j \in J$. Finally, we conclude that 
\begin{eqnarray*}
	e_{X,Y}(\{\alpha_{i}\}_{i \in I} \times E) \delta_{2} e_{X,Y}(\{\beta_{j}\}_{j \in J} \times F).
\end{eqnarray*}

\begin{example}
	Consider the proximal evaluation map $e_{I,X} : (X^{I} \times I,\delta) \rightarrow (X,\delta_{1})$. Since $X^{I} \times \{0\}$ is proximally isomorphic to $X^{I}$ by the map $(\alpha,0) \mapsto \alpha(0)$, the restriction \[e^{0}_{I,X} = e_{I,X}|_{(X^{I} \times \{0\})} : (X^{I},\delta^{'}) \rightarrow (X,\delta_{1}),\]
	defined by $e^{0}_{I,X}(\alpha) = \alpha(0)$, is a pc-map. 
\end{example}

\begin{example}
	Let $e_{I,X \times X} : ((X \times X)^{I} \times I,\delta) \rightarrow (X,\delta_{1})$ be the proximal evaluation map. By Theorem \ref{t1}, the restriction \[e^{0}_{I,X \times X} = e_{I,X \times X}|_{(X^{I} \times \{0\})} : (X^{I},\delta^{'}) \rightarrow (X \times X,\delta^{'}),\]
	defined by $e^{0}_{I,X \times X}(\alpha) = (\alpha(0),\alpha(1))$, is a pc-map. 
\end{example}

\quad Note that, in topological spaces, the map $X^{I} \rightarrow X \times X$, $\alpha \mapsto (\alpha(0),\alpha(1))$, is the path fibration. Similarly, the map $X^{I} \rightarrow X$, $\alpha \mapsto \alpha(0)$, is the path fibration with a fixed initial point at $t = 0$.

\subsection{Proximal Covering Spaces}
\label{subsec:2}

A covering space of a topological space and the fundamental group are tightly related. One can categorize all the covering spaces of a topological space using the subgroups of its fundamental group. Covering spaces are not only useful in algebraic topology, but also in complex dynamics, geometric group theory, and the theory of Lie groups.  

\begin{definition}\label{d1}
	A surjective and pc-map $p : (X,\delta) \rightarrow (X^{'},\delta^{'})$ is a proximal covering map if the following hold:
	\begin{itemize}
		\item Let $\{x^{'}\} \subseteq X^{'}$ be any subset with $\{x^{'}\} \ll_{\delta^{'}} Y^{'}$. Then there is an index set $I$ satisfying that
		\begin{eqnarray*}
			p^{-1}(Y^{'}) = \displaystyle\bigcup_{i \in I}Y_{i}
		\end{eqnarray*}
	with $V_{i} \ll_{\delta} Y_{i}$, where $V_{i} \in p^{-1}(\{x^{'}\})$ for each $i \in I$.
	    \item $Y_{i} \neq Y_{j}$ when $i \neq j$ for $i$, $j \in I$.
		\item $p|_{Y_{i}} : Y_{i} \rightarrow Y^{'}$ is a proximal isomorphism for every $i \in I$.  
	\end{itemize}
\end{definition}

\quad In Definition \ref{d1}, $(X,\delta)$ is called a proximal covering space of $(X^{'},\delta^{'})$. For $i \in I$, $Y_{i}$ is said to be a proximal sheet. For any $x^{'} \in X^{'}$, $p^{-1}(\{x^{'}\})$ is called a proximal fiber of $x^{'}$. The map $p|_{Y_{i}} : Y_{i} \rightarrow Y^{'}$ is a proximal isomorphism if the map $p : (X,\delta) \rightarrow (X^{'},\delta^{'})$ is a proximal isomorphism. However, the converse is not generally true. Given any proximity $\delta$ on $X$, it is obvious that the identity map on $X$ is always a proximal covering map. 

\begin{figure}[h]
	\centering
	\includegraphics[width=0.80\textwidth]{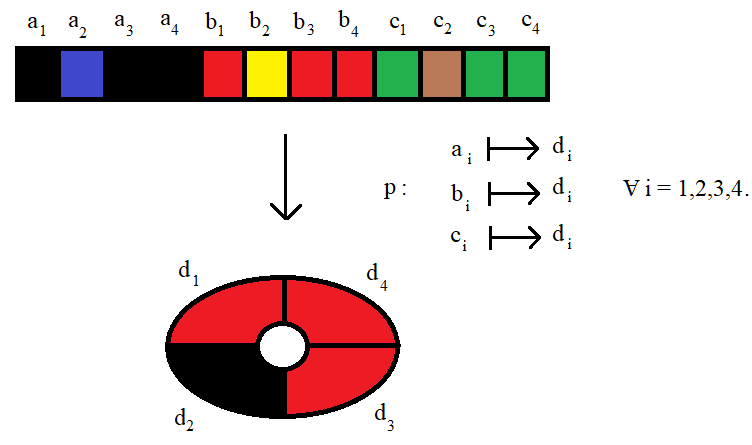}
	\caption{A map $p$ from $\{a_{1},a_{2},a_{3},a_{4}\} \cup \{b_{1},b_{2},b_{3},b_{4}\} \cup \{c_{1},c_{2},c_{3},c_{4}\}$ to $\{d_{1},d_{2},d_{3},d_{4}\}$ defined by $p(a_{i}) = p(b_{i}) = p(c_{i}) = d_{i}$ for any $i = 1,2,3,4$.}
	\label{fig:2}
\end{figure}

\begin{example}
	Assume that $X = \{a_{1},a_{2},a_{3},a_{4}\} \cup \{b_{1},b_{2},b_{3},b_{4}\} \cup \{c_{1},c_{2},c_{3},c_{4}\}$ and $X^{'} = \{d_{1},d_{2},d_{3},d_{4}\}$ are two proximity spaces such that $p : (X,\delta) \rightarrow (X^{'},\delta^{'})$ is a surjective and pc-map defined by $p(a_{i}) = p(b_{i}) = p(c_{i}) = d_{i}$ for each $i = 1,2,3,4$ (see Figure \ref{fig:2}). Let $\{d_{1}\} \subset X^{'}$ and $Y^{'} = \{d_{1},d_{2},d_{4}\}$ a proximal $\delta^{'}-$neighborhood of $\{d_{1}\}$. For $V_{1} = \{a_{1}\}$, $V_{2} = \{b_{1}\}$, and $V_{3} = \{c_{1}\}$, we have $p^{-1}(Y^{'}) = \displaystyle \bigcup_{i=1}^{3}Y_{i}$, where $Y_{1} = \{a_{1},a_{2},a_{4}\}$, $V_{2} = \{b_{1},b_{2},b_{4}\}$, and $V_{3} = \{c_{1},c_{2},c_{4}\}$. Note that $V_{1} \underline{\delta} (X-Y_{1})$, $V_{2} \underline{\delta} (X-Y_{2})$, and $V_{3} \underline{\delta} (X-Y_{3})$, i.e., $Y_{i}$ is a proximal $\delta-$neighborhood of $V_{i}$ for each $i \in \{1,2,3\}$. Moreover, for $i$, $j \in \{1,2,3\}$ with $i \neq j$, we have that $Y_{i}$ is not $Y_{j}$, and $p|_{Y_{1}} : \{a_{1},a_{2},a_{4}\} \rightarrow \{d_{1},d_{2},d_{4}\}$, $p|_{Y_{2}} : \{b_{1},b_{2},b_{4}\} \rightarrow \{d_{1},d_{2},d_{4}\}$, and $p|_{Y_{3}} : \{c_{1},c_{2},c_{4}\} \rightarrow \{d_{1},d_{2},d_{4}\}$ are proximal isomorphisms. For other points $d_{2}$, $d_{3}$, and $d_{4}$, a similar process is done. This shows that $p$ is a proximal covering map.
\end{example}

\begin{example}
	Given a proximity $\delta$ on $X$, consider the surjective and pc-map $p : (X \times \{0,1,2\cdots\},\delta^{'}) \rightarrow (X,\delta)$ with $p(x,t) = x$. For a proximal $\delta-$neighborhood $Y$ of any subset $\{x\} \subset X$, we have that \[p^{-1}(Y) = Y \times \mathbb{Z}^{+} \subset X \times \{0,1,2\cdots\}\]
	for a proximal $\delta^{'}-$neighborhood $Y^{'}$ of $Y \times \mathbb{Z}^{+}$.	Moreover, $p|_{Y \times \mathbb{Z}^{+}} : Y \times \mathbb{Z}^{+} \rightarrow Y$, $p|_{Y \times \mathbb{Z}^{+}}(x,t) = x$, is a proximal isomorphism. Thus, $p$ is a proximal covering map.
\end{example}

\begin{proposition}\label{prop1}
	Any proximal isomorphism is a proximal covering map. 
\end{proposition}

\begin{proof}
	Let $p : (X,\delta) \rightarrow (X^{'},\delta^{'})$ be a proximal isomorphism. Then $p$ is a pc-map. Lemma 4.1 of \cite{NaimpallyPeters:2013} says that $p$ is continuous with respect to compatible topologies. Therefore, we get $p^{-1}(Y^{'}) = Y \subset X$ for an open neighborhood $Y^{'}$ of any subset $\{x^{'}\} \subseteq X^{'}$. Combining with the fact that a proximal neighborhood of a set is also a neighborhood, we conclude that $p^{-1}(Y^{'}) = Y$ for a proximal $\delta^{'}-$neighborhood $Y^{'}$ of $\{x^{'}\}$ in $X^{'}$ and a proximal $\delta-$neighborhood $Y$ of $V$, where $V \in p^{-1}(\{x^{'}\})$ in $X$. Furthermore, $p|_{Y} : Y \rightarrow Y^{'}$ is an isomorphism of proximity spaces because $p$ is an isomorphism of proximity spaces. Finally, $p$ is a proximal covering map.
\end{proof}

\quad The following diagram illustrates two ways to prove that any proximal isomorphism $p : (X,\delta) \rightarrow (X^{'},\delta^{'})$ is a covering map between respective compatible topologies on both $(X,\delta)$ and $(X^{'},\delta^{'})$:
\begin{displaymath}
	\xymatrix{
		\text{proximal isomorphism} \ar[r] \ar[d] &
		\text{proximal covering map} \ar[d] \\
		\text{homeomorphism} \ar[r] & \text{covering map}.}
\end{displaymath}

\begin{theorem}
	The cartesian product of two proximal covering maps is a proximal covering map.
\end{theorem}

\begin{proof}
	Let $p : (X,\delta_{1}) \rightarrow (X^{'},\delta_{1}^{'})$ and $q : (Y,\delta_{2}) \rightarrow (Y^{'},\delta_{2}^{'})$ be two proximal covering maps. Then for a proximal $\delta_{1}^{'}-$neighborhood $M_{1}^{'}$ of $\{x_{1}^{'}\} \subset X^{'}$, we have that
	\begin{eqnarray*}
		p^{-1}(M_{1}^{'}) = \displaystyle\bigcup_{i \in I}M_{i}
	\end{eqnarray*}
    for a proximal $\delta_{1}-$neighborhood $M_{1}$ of $V_{i}$, where $V_{i} \in p^{-1}(\{x_{1}^{'}\})$. We also have that $M_{i} \neq M_{k}$ with any $k \in I$ when $i \neq k$. Similarly, for a proximal $\delta_{2}^{'}-$neighborhood $N_{2}^{'}$ of $\{x_{2}^{'}\} \subset Y^{'}$, we have that
    \begin{eqnarray*}
    	q^{-1}(N_{2}^{'}) = \displaystyle\bigcup_{j \in J}N_{j}
    \end{eqnarray*}
    for a proximal $\delta_{2}-$neighborhood $N_{j}$ of $W_{j}$, where $W_{j} \in q^{-1}(\{x_{2}^{'}\})$. Also, we have that $N_{i} \neq N_{l}$ with any $l \in J$ when $j \neq l$. For a proximal neighborhood $M_{1}^{'} \times N_{2}^{'}$ of $\{x_{1}^{'}\} \times \{x_{2}^{'}\} \subset X^{'} \times Y^{'}$, we get
    \[(p \times q)^{-1}(M_{1}^{'} \times N_{2}^{'}) = p^{-1}(M_{1}^{'}) \times q^{-1}(N_{2}^{'}) = \displaystyle\bigcup_{i \in I}M_{i} \times \displaystyle\bigcup_{j \in J}N_{j} = \displaystyle \bigcup_{\substack{{i \in I} \\ {j \in J}}}(M_{i} \times N_{j}).\]
    It is clear that $M_{i} \times N_{j} \neq M_{k} \times N_{l}$ when $(i,j) \neq (k,l)$ for any $i$, $k \in I$ and $j$, $l \in J$. Moreover, since $p|_{M_{i}} : M_{i} \rightarrow M_{1}^{'}$ and $q|_{N_{j}} : N_{j} \rightarrow N_{2}^{'}$ are proximal isomorphisms, the map $(p \times q)|_{M_{i} \times N_{j}} : M_{i} \times N_{j} \rightarrow M_{1}^{'} \times N_{2}^{'}$ is a proximal isomorphism. Consequently, $p \times q : X \times Y \rightarrow X^{'} \times Y^{'}$ is a proximal covering map.
\end{proof}

\subsection{Proximal Fibrations}
\label{subsec:3}

Some topological problems can be conceptualized as lifting or extension problems. In the homotopy-theoretic viewpoint, fibrations and cofibrations deal with them, respectively (see Section \ref{subsec:4} for the detail of cofibrations). Postnikov systems, spectral sequences, and obstruction theory, which are important tools constructed on homotopy theory, involve fibrations. On the other hand, the notion of proximal fibration of proximity spaces is first mentioned in \cite{PetersTane2:2022}, and we extend this with useful properties in proximity cases.

\begin{definition}
	A pc-map $p : (X,\delta) \rightarrow (X^{'},\delta^{'})$ is said to have the proximal homotopy lifting property (PHLP) with respect to a proximity space $(X^{''},\delta^{''})$ if for an inclusion map $i_{0} : (X^{''},\delta^{''}) \rightarrow (X^{''} \times I,\delta_{1})$, $i_{0}(x^{''}) = (x^{''},0)$, for every pc-map $k : (X^{''},\delta^{''}) \rightarrow (X,\delta)$, and prox-hom $G : (X^{''} \times I,\delta_{1}) \rightarrow (X^{'},\delta^{'})$ with $p \circ k = G \circ i_{0}$, then there exists a prox-hom $G^{'} : (X^{''} \times I,\delta_{1}) \rightarrow (X,\delta)$ for which $G^{'}(x^{''},0) = k(x^{''})$ and $p \circ G^{'}(x^{''},t) = G(x^{''},t)$.
	\begin{displaymath}
		\xymatrix{
			X^{''} \ar[r]^{k} \ar[d]_{i_{0}} &
			X \ar[d]^{p} \\
			X'' \times I \ar[r]_{G} \ar@{.>}[ur]^{G^{'}} & X^{'}. }
	\end{displaymath}
\end{definition}

\begin{definition}
	A pc-map $p : (X,\delta) \rightarrow (X^{'},\delta^{'})$ is said to be a proximal fibration if it has the PHLP for any proximity space $(X^{''},\delta^{''})$.
\end{definition}

\begin{example}
	For any proximity spaces $(X,\delta)$ and $(X^{'},\delta^{'})$, we shall show that the projection map $\pi_{1} : (X \times X^{'},\delta_{2}) \rightarrow (X,\delta)$ onto the first factor is a proximal fibration. Consider the diagram
	\begin{displaymath}
		\xymatrix{
			X^{''} \ar[r]^{(k_{X},k_{X^{'}})} \ar[d]_{i_{0}} &
			X \times X^{'} \ar[d]^{\pi_{1}} \\
			X'' \times I \ar[r]_{G} \ar@{.>}[ur]^{G^{'}} & X^{'}}
	\end{displaymath}
    with $\pi_{1} \circ (k_{X},k_{X^{'}}) = G \circ i_{0}$. Then there is a map $G^{'} : (X^{''} \times I,\delta_{1}) \rightarrow (X \times X^{'},\delta_{2})$ defined by $G^{'} = (G,F)$, where $F : (X^{''} \times I,\delta_{1}) \rightarrow (X^{'},\delta^{'})$ is the composition of the first projection map $(X^{''} \times I,\delta_{1}) \rightarrow (X^{''},\delta^{''})$ and $k_{X^{'}}$. Since $k_{X^{'}}$ and the first projection map are pc-maps, it follows that $F$ is a pc-map. Moreover, we get $F(x^{''},0) = F(x^{''},1) = k_{X^{'}}(x^{''})$, which means that $F$ is a (constant) prox-hom. Combining this result with the fact that $H$ is a prox-hom, we have that $G^{'}$ is a prox-hom. Moreover, we get
    \begin{eqnarray*}
    	G^{'} \circ i(x^{''}) &=& G^{'}(x^{''},0) = (G(x^{''},0),F(x^{''},0)) = (k_{X}(x^{''}),k_{X^{'}}(x^{''}))\\
    	&=& (k_{X},k_{X^{'}})(x^{''}),
    \end{eqnarray*}
    and
    \begin{eqnarray*}
    	\pi_{1} \circ G^{'}(x^{''},t) = \pi_{1}(G(x^{''},t),F(x^{''},t)) = G(x^{''},t).
    \end{eqnarray*}
    This shows that $\pi_{1}$ is a proximal fibration.
\end{example}

\begin{example}
	Let $c : (X,\delta) \rightarrow (\{x_{0}\},\delta_{0})$ be the constant map of proximity spaces. Given the diagram
	\begin{displaymath}
		\xymatrix{
			X^{''} \ar[r]^{k} \ar[d]_{i_{0}} &
			X \ar[d]^{p} \\
			X'' \times I \ar[r]_{G} \ar@{.>}[ur]^{G^{'}} & \{x_{0}\}}
	\end{displaymath}
    with the condition $p \circ k(x^{''}) = G \circ i_{0}(x^{''}) = \{x_{0}\}$. Then there exists a (constant) prox-hom $G^{'} : (X^{''} \times I,\delta_{1}) \rightarrow (X,\delta)$ defined by $G^{'}(x^{''},t) = k(x^{''})$ satisfying that
    \begin{eqnarray*}
    	&&p \circ G^{'}(x^{''},t) = p(G^{'}(x^{''},t)) = \{x_{0}\} = G(x^{''},t),\\
    	&&G^{'} \circ i_{0}(x^{''}) = G^{'}(x^{''},0) = k(x^{''}).
    \end{eqnarray*}
    This proves that $p$ is a proximal fibration.
\end{example}

\begin{proposition}
	\textbf{i)} The composition of two proximal fibrations is also a proximal fibration.
	
	\textbf{ii)} The cartesian product of two proximal fibrations is also a proximal fibration. 
\end{proposition}

\begin{proof}
	\textbf{i)} Let $p_{1} : (X_{1},\delta_{1}) \rightarrow (Y_{1},\delta_{1}^{'})$ and $p_{2} : (Y_{1},\delta_{1}^{'}) \rightarrow (Y_{2},\delta_{2}^{'})$ be any proximal fibrations. Then for the inclusion map $i_{0} : (X^{''},\delta^{''}) \rightarrow (X^{''} \times I,\delta_{3})$, pc-maps \linebreak$k_{1} : (X^{''},\delta^{''}) \rightarrow (X_{1},\delta_{1})$, $k_{2} : (X^{''},\delta^{''}) \rightarrow (Y_{1},\delta_{1}^{'})$, and proximal homotopies \linebreak$G_{1} : (X^{''} \times I,\delta_{3}) \rightarrow (Y_{1},\delta_{1}^{'})$, $G_{2} : (X^{''} \times I,\delta_{3}) \rightarrow (Y_{2},\delta_{2}^{'})$ with the property $p_{1} \circ k_{1} = G_{1} \circ i_{0}$ and $p_{2} \circ k_{2} = G_{2} \circ i_{0}$, there exist two proximal homotopies $G_{1}^{'} : (X^{''} \times I,\delta_{3}) \rightarrow (X_{1},\delta_{1})$ and $G_{2}^{'} : (X^{''} \times I,\delta_{3}) \rightarrow (Y_{1},\delta_{1}^{'})$ satisfying that
	\begin{eqnarray*}
		&& G_{1}^{'} \circ i_{0} = k_{1}, p_{1} \circ G_{1}^{'} = G_{1},\\
		&& G_{2}^{'} \circ i_{0} = k_{2}, p_{2} \circ G_{2}^{'} = G_{2}.
	\end{eqnarray*}
    If we take $G_{2}^{'} = G_{1}$, then we have the following commutative diagram:
    \begin{displaymath}
    	\xymatrix{
    		X^{''} \ar[r]^{k_{1}} \ar[d]_{i_{0}} &
    		X_{1} \ar[d]^{p_{2} \circ p_{1}} \\
    		X'' \times I \ar[r]_{G_{2}} \ar@{.>}[ur]^{G_{1}^{'}} & Y_{2}.}
    \end{displaymath}
    Thus, we get
    \begin{eqnarray*}
    	&&G_{1}^{'} \circ i_{0} = k_{1},\\
    	&&(p_{2} \circ p_{1}) \circ G_{1}^{'} = G_{2}.
    \end{eqnarray*}
    This show that the composition $p_{2} \circ p_{1}$ is a proximal fibration.
    
    \textbf{ii)} Let $p_{1} : (X_{1},\delta_{1}) \rightarrow (Y_{1},\delta_{1}^{'})$ and $p_{2} : (X_{2},\delta_{2}) \rightarrow (Y_{2},\delta_{2}^{'})$ be any proximal fibrations. Then for the inclusion map $i_{0} : (X^{''},\delta^{''}) \rightarrow (X^{''} \times I,\delta_{3})$, pc-maps $k_{1} : (X^{''},\delta^{''}) \rightarrow (X_{1},\delta_{1})$, $k_{2} : (X^{''},\delta^{''}) \rightarrow (X_{2},\delta_{2})$, and proximal homotopies $G_{1} : (X^{''} \times I,\delta_{3}) \rightarrow (Y_{1},\delta_{1}^{'})$, $G_{2} : (X^{''} \times I,\delta_{3}) \rightarrow (Y_{2},\delta_{2}^{'})$ with the property $p_{1} \circ k_{1} = G_{1} \circ i_{0}$ and $p_{2} \circ k_{2} = G_{2} \circ i_{0}$, there exist two proximal homotopies $G_{1}^{'} : (X^{''} \times I,\delta_{3}) \rightarrow (X_{1},\delta_{1})$ and $G_{2}^{'} : (X^{''} \times I,\delta_{3}) \rightarrow (X_{2},\delta_{2})$ satisfying that
    \begin{eqnarray*}
    	&& G_{1}^{'} \circ i_{0} = k_{1}, p_{1} \circ G_{1}^{'} = G_{1},\\
    	&& G_{2}^{'} \circ i_{0} = k_{2}, p_{2} \circ G_{2}^{'} = G_{2}.
    \end{eqnarray*}
    Consider the map $G_{3}^{'} = (G_{1}^{'},G_{2}^{'})$. Then $G_{3}^{'}$ is clearly a prox-hom and we have the following commutative diagram:
    \begin{displaymath}
    	\xymatrix{
    		X^{''} \ar[r]^{(k_{1},k_{2})} \ar[d]_{i_{0}} &
    		X_{1} \times X_{2} \ar[d]^{p_{1} \times p_{2}} \\
    		X'' \times I \ar[r]_{(G_{1},G_{2})} \ar@{.>}[ur]^{G_{3}^{'}} & Y_{1} \times Y_{2}.}
    \end{displaymath}
    Thus, we get
    \begin{eqnarray*}
    	&&G_{3}^{'} \circ i_{0} = (k_{1},k_{2}),\\
    	&&(p_{1} \times p_{2}) \circ G_{3}^{'} = (G_{1},G_{2}).
    \end{eqnarray*}
    This proves that the cartesian product $p_{1} \times p_{2}$ is a proximal fibration.
\end{proof}

\quad Let $f : (X,\delta_{1}) \rightarrow (Y,\delta_{2})$ be a pc-map. Then for any pc-map $g : (Z,\delta_{3}) \rightarrow (Y,\delta_{2})$, a proximal lifting of $f$ is a pc-map $h : (X,\delta_{1}) \rightarrow (Z,\delta_{3})$ satisfying that $f = g \circ h$. 

\begin{proposition}
	Let $p : (X,\delta_{1}) \rightarrow (Y,\delta_{2})$ be a proximal fibration. Then
	
	\textbf{i)} The pullback $g^{\ast}p : (P,\delta) \rightarrow (Y^{'},\delta_{2}^{'})$ is a proximal fibration for any pc-map $g : (Y^{'},\delta_{2}^{'}) \rightarrow (Y,\delta_{2})$.
	\begin{displaymath}
		\xymatrix{
			P \ar[r]^{\pi_{1}} \ar[d]_{g^{\ast}p} &
			X \ar[d]^{p} \\
			Y^{'} \ar[r]_{g} & Y.}
	\end{displaymath}

    \textbf{ii)} For any proximity space $(Z,\delta_{3})$, the map $p_{\ast} : (X^{Z},\delta_{3}^{'}) \rightarrow (Y^{Z},\delta_{3}^{''})$ is a proximal fibration.
\end{proposition}

\begin{proof}
	\textbf{i)} Let \[P = \{(x,y^{'}) \ | \ g(y^{'}) = p(e)\} \subseteq X \times Y^{'}\] be a proximity space with the proximity $\delta_{0}$ on itself. Since $p$ is a proximal fibration, for an inclusion map $i_{0} : (X^{''},\delta^{''}) \rightarrow (X^{''} \times I,\delta_{3})$, for any pc-map $k_{1}$ from $(X^{''},\delta^{''})$ to $(X,\delta_{1})$, and prox-hom $G_{1} : (X^{''} \times I,\delta_{3}) \rightarrow (X,\delta)$ with $p \circ k_{1} = G_{1} \circ i_{0}$, there exists a prox-hom \[G_{1}^{'} : (X^{''} \times I,\delta_{3}) \rightarrow (X,\delta_{1})\] for which $G_{1}^{'}(x^{''},0) = k_{1}(x^{''})$ and $p \circ G_{1}^{'}(x^{''},t) = G_{1}(x^{''},t)$. Assume that a map $k_{2}$ from $(X^{''},\delta^{''})$ to $(P,\delta_{0})$ is a pc-map and $G_{2} : (X^{''} \times I,\delta_{3}) \rightarrow (Y^{'},\delta_{2}^{'})$ is a prox-hom with $g^{\ast}p \circ k_{2} = G_{2} \circ i_{0}$. If we define $G_{2}^{'} : (X^{''} \times I,\delta_{3}) \rightarrow (P,\delta_{0})$ by $G_{2}^{'} = (G_{1}^{'},G_{2})$, then we observe that
	\begin{eqnarray*}
		&&G_{2}^{'} \circ i_{0} = k_{2},\\
		&&g^{\ast}p \circ G_{2}^{'} = G_{2}.
	\end{eqnarray*} 
    This gives the desired result.
    
    \textbf{ii)} Consider the following diagrams:
    \begin{displaymath}
    	\xymatrix{
    		Z \times X^{''} \ar[r]^{k_{1}} \ar[d]_{i_{0}} &
    		X \ar[d]^{p} \\
    		Z \times X'' \times I \ar[r]_{G_{1}} \ar@{.>}[ur]^{G_{1}^{'}} & Y}
    \end{displaymath}
    and
    \begin{displaymath}
    	\xymatrix{
    		X^{''} \ar[r]^{k_{2}} \ar[d]_{i_{0}} &
    		X^{Z} \ar[d]^{p_{\ast}} \\
    		X'' \times I \ar[r]_{G_{2}} \ar@{.>}[ur]^{G_{2}^{'}} & Y^{Z}.}
    \end{displaymath}
    Since $p$ is a proximal fibration, we have $H_{1}^{'} : (Z \times X'' \times I,\delta^{''}) \rightarrow (X,\delta_{1})$ as the prox-hom in the upper diagram. $Z \times X'' \times I$ is proximally isomorphic to $X^{''} \times I \times Z$ and we can think of $G_{1}^{'}$ as the prox-hom $(X^{''} \times I \times Z,\delta^{''}) \rightarrow (X,\delta_{1})$. By Proposition \ref{p1}, we have the prox-hom $G_{2}^{'} : (X^{''} \times I,\delta^{'}) \rightarrow (X^{Z},\delta_{3}^{'})$ in the lower diagram. This map satisfies the desired conditions, and thus, we conclude that $p_{\ast}$ is a proximal fibration.
\end{proof}

\subsection{Proximal Cofibrations}
\label{subsec:4}

Similar to the proximal fibration, we currently deal with the notion of proximal cofibration of proximity spaces. We first study the problem of extension in homotopy theory, and then present the definition of proximal cofibration with its basic results.

\begin{definition}\label{d4}
	Given two proximity spaces $(X,\delta)$ and $(X^{'},\delta^{'})$, a pc-map \linebreak$h : X \rightarrow X^{'}$ is said to have a proximal homotopy extension property (PHEP) with respect to a proximity space $(X^{''},\delta^{''})$ if for inclusion maps $i_{0}^{X} : (X,\delta) \rightarrow (X \times I,\delta_{1})$ and $i_{0}^{X^{'}} : (X^{'},\delta^{'}) \rightarrow (X^{'} \times I,\delta_{1}^{'})$, for every pc-map $k : (X^{'},\delta^{'}) \rightarrow (X^{''},\delta^{''})$, and prox-hom $F : (X \times I,\delta_{1}) \rightarrow (X^{''},\delta^{''})$ with $k \circ (h \times 1_{0}) = F \circ i_{0}^{X}$, then there exists a prox-hom $F^{'} : (X^{'} \times I,\delta_{1}^{'}) \rightarrow (X^{''},\delta^{''})$ satisfying $F^{'} \circ i_{0}^{X^{'}} = k$ and $F^{'} \circ (h \times 1_{I}) = F$.
\end{definition}
\begin{displaymath}
	\xymatrix{X \times 0 \ar@{^{(}->}[rr]^{i_{0}^{X}} \ar[dd]_{h \times 1_{0}} & & X \times I \ar[dd]^{h \times 1_{I}} \ar[dl]_{F} \\
		& X^{''} &\\
		X^{'} \times 0 \ar@{^{(}->}[rr]_{i_{0}^{X^{'}}} \ar[ur]^{k} & & X^{'} \times I. \ar@{.>}[ul]_{F^{'}}}
\end{displaymath}

\begin{figure}[h]
	\centering
	\includegraphics[width=0.50\textwidth]{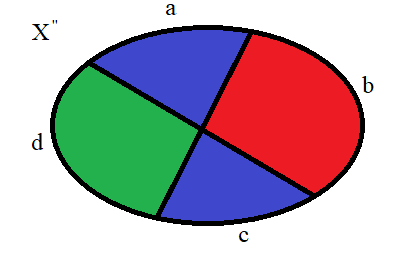}
	\caption{The picture is represented by $X^{''} = \{a,b,c,d\}$.}
	\label{fig:3}
\end{figure}

\begin{example}
	Let $X^{''} = \{a,b,c,d\}$ be a set with the proximity $\delta^{''}$ on itself as in Figure \ref{fig:3}. Let $\gamma_{1}$ and $\gamma_{2}$ be proximal paths on $X^{''}$ such that $\gamma_{1}(0) = b$, $\gamma_{1}(1) = a$, $\gamma_{2}(0) = b$, and $\gamma_{2}(1) = c$. Consider the following diagram for an inclusion map $h : (\{0\},\delta) \rightarrow (I,\delta^{'})$, where $k : (I,\delta^{'}) \rightarrow (X^{''},\delta^{''})$ is the map $\gamma_{2}$ and $F : (\{0\} \times I,\delta_{1}) \rightarrow (X^{''},\delta^{''})$ is defined by $F(0,t) = \gamma_{1}(t)$ for all $t \in I$:
	\begin{displaymath}
		\xymatrix{\{0\} \times 0 \ar@{^{(}->}[rr]^{i_{0}^{\{0\}}} \ar@{^{(}->}[dd]_{h \times 1_{0}} & & \{0\} \times I \ar@{^{(}->}[dd]^{h \times 1_{I}} \ar[dl]_{F} \\
			& X^{''} &\\
			I \times 0 \ar@{^{(}->}[rr]_{i_{0}^{I}} \ar[ur]^{k} & & I \times I,}
	\end{displaymath}
    i.e., the equality $k \circ (h \times 1_{0}) = F \circ i_{0}^{\{0\}}$ holds. Then, by Gluing Lemma, there exists a prox-hom $F^{'} : (I \times I,\delta_{1}^{'}) \rightarrow (X^{''},\delta^{''})$ defined by $F^{'}(0,t_{1}) = F(0,t_{1})$ and $F^{'}(t_{2},0) = k(t_{2})$ for all $(t_{1},t_{2}) \in I \times I$ which satisfy
    \begin{eqnarray*}
    	&&F^{'} \circ (h \times 1_{I}) = F,\\
    	&&F^{'} \circ i_{0}^{I} = k.
    \end{eqnarray*}
    Schematically, we have the diagram
    \begin{displaymath}
    	\xymatrix{\{0\} \times 0 \ar@{^{(}->}[rr]^{i_{0}^{\{0\}}} \ar@{^{(}->}[dd]_{h \times 1_{0}} & & \{0\} \times I \ar@{^{(}->}[dd]^{h \times 1_{I}} \ar[dl]_{F} \\
    		& X^{''} &\\
    		I \times 0 \ar@{^{(}->}[rr]_{i_{0}^{I}} \ar[ur]^{k} & & I \times I. \ar@{.>}[ul]_{F^{'}}}
    \end{displaymath}
\end{example}

\begin{definition}
	A pc-map $h : (X,\delta) \rightarrow (X^{'},\delta^{'})$ is said to be a proximal cofibration if it has the PHEP with respect to any proximity space $(X^{''},\delta^{''})$.
\end{definition}

\begin{example}
	Let $h : (X,\delta) \hookrightarrow (X^{'},\delta^{'})$ be an inclusion map such that $X \subset X^{'}$. Then $h$ is a natural proximal cofibration since there exists a prox-hom \[F^{'} = F|_{X^{'}} : (X^{'} \times I,\delta_{1}^{'}) \rightarrow (X^{''},\delta^{''})\] satisfying the conditions of PHEP with respect to any proximity space $(X^{''},\delta^{''})$. 
\end{example}

\begin{proposition}
	\textbf{i)} Let $h : (X,\delta) \rightarrow (X^{'},\delta^{'})$ and $h^{'} : (Y,\delta_{1}) \rightarrow (Y^{'},\delta_{1}^{'})$ be two maps such that $X$ and $X^{'}$ are proximally isomorphic to $Y$ and $Y^{'}$, respectively, and the following diagram commutes:
	\begin{displaymath}
		\xymatrix{
			X \ar[r]^{\approx_{\delta}} \ar[d]_{h} &
			Y \ar[d]^{h^{'}} \\
			X^{'} \ar[r]_{\approx_{\delta}} & Y^{'}.}
	\end{displaymath}
    Then $h$ is a proximal cofibration if and only if $h^{'}$ is a proximal cofibration.
    
    \textbf{ii)} The composition of two proximal cofibrations is also a proximal cofibration.
    
    \textbf{iii)} The coproduct of two proximal cofibrations is also a proximal cofibration.
    
    \textbf{iv)} Let $h : (X,\delta) \rightarrow (X^{'},\delta^{'})$ be a proximal cofibration and the following is a pushout diagram.
    \begin{displaymath}
    	\xymatrix{
    		X \ar[r]^{l} \ar[d]_{h} &
    		Y \ar[d]^{h^{'}} \\
    		X^{'} \ar[r]_{l^{'}} & Y^{'}.}
    \end{displaymath}
    Then $h^{'}$ is a proximal cofibration.
\end{proposition}

\begin{proof}
	\textbf{i)} Let $h$ be a proximal cofibration. By  Definition \ref{d4}, there is a prox-hom $F^{'} : (X^{'} \times I,\delta_{2}^{'}) \rightarrow (X^{''},\delta^{''})$ such that \[F^{'} \circ i_{0}^{X^{'}} = k \ \ \text{and} \ \ F^{'} \circ (h \times 1_{I}) = F\] for any pc-map $k : (X^{'},\delta^{'}) \rightarrow (X^{''},\delta^{''})$ and prox-hom $F$ from $(X \times I,\delta_{2})$ to $(X^{''},\delta^{''})$ with $k \circ (h \times 1_{0}) = F \circ i_{0}^{X}$. Assume that $\beta_{1} : X \rightarrow Y$ and $\beta_{2} : X^{'} \rightarrow Y^{'}$ are two proximal isomorphisms. Since the diagram commutes, we know that \[h^{'} \circ \beta_{1} = \beta_{2} \circ h.\] Let $i_{0}^{Y} : (Y,\delta_{1}) \rightarrow (Y \times I,\delta_{3})$ and $i_{0}^{Y^{'}} : (Y^{'},\delta_{1}^{'}) \rightarrow (Y^{'} \times I,\delta_{3}^{'})$ be two inclusion maps, $k^{'} := k \circ (\beta_{2})^{-1} : (Y^{'},\delta_{1}^{'}) \rightarrow (X^{''},\delta^{''})$ a pc-map, and $F^{''} := F^{'} \circ (\beta_{1}^{-1} \times 1_{I})$ from $(Y \times I,\delta_{3})$ to $(X^{''},\delta^{''})$ a prox-hom for which \[k^{'} \circ (h^{'} \times 1_{0}) = F^{'} \circ i_{0}^{Y}.\] Then there exists a prox-hom \[F^{'''} := F^{'} \circ ((\beta_{2})^{-1} \times 1_{I}) : (Y^{'} \times I,\delta_{3}^{'}) \rightarrow (X^{''},\delta^{''})\] such that $F^{''} \circ i_{0}^{Y^{'}} = k^{'}$ and $F^{''} \circ (h^{'} \times 1_{I}) = F^{'}$.
	Conversely, assume that $h^{'}$ is a proximal cofibration. Similarly, for a prox-hom $F^{'}$ from $(Y^{'} \times I,\delta_{3}^{'})$ to $(X^{''},\delta^{''})$ that makes $h^{'}$ a cofibration, there exists a prox-hom \[F^{''} := F^{'} \circ (f^{'} \times 1_{I}): (X^{'} \times I,\delta_{2}^{'}) \rightarrow (X^{''},\delta^{''})\] that makes $h$ a proximal cofibration.
	
	\textbf{ii)} Let $h : (X,\delta) \rightarrow (X^{'},\delta^{'})$ and $h^{'} : (X^{'},\delta^{'}) \rightarrow (Y,\delta_{1})$ be two proximal cofibrations. Then for any pc-map $k : (X^{'},\delta^{'}) \rightarrow (X^{''},\delta^{''})$ and prox-hom $F$ from $(X \times I,\delta_{2})$ to $(X^{''},\delta^{''})$ with \[k \circ (h \times 1_{0}) = F \circ i_{0}^{X},\] there is a prox-hom $F^{'} : (X^{'} \times I,\delta_{2}^{'}) \rightarrow (X^{''},\delta^{''})$ such that $F^{'} \circ i_{0}^{X^{'}} = k$ and $F^{'} \circ (h \times 1_{I}) = F$, similarly, for any pc-map $k^{'} : (Y,\delta_{1}) \rightarrow (X^{''},\delta^{''})$ and prox-hom $G^{'} : (X^{'} \times I,\delta_{2}^{'}) \rightarrow (X^{''},\delta^{''})$ with \[k^{'} \circ (h^{'} \times 1_{0}) = G^{'} \circ i_{0}^{X^{'}},\] there is a prox-hom $F^{''} : (Y \times I,\delta_{3}^{'}) \rightarrow (X^{''},\delta^{''})$ such that $F^{''} \circ i_{0}^{Y} = k^{'}$ and $F^{''} \circ (h^{'} \times 1_{I}) = G^{'}$. Combining these results with the fact $F^{'} = G^{'}$, we have the following: For a pc-map $k^{''} := k^{'}$ and prox-hom $G^{''} := F$ with \[k^{''} \circ ((h^{'} \circ h) \times 1_{0}) = G^{''} \circ i_{0}^{X},\] there is a prox-hom $F^{'''} := F^{''}$ such that \[F^{'''} \circ i_{0}^{Y} = k^{''} \ \ \text{and} \ \ F^{'''} \circ ((h^{'} \circ h) \times 1_{I}) = G^{''}.\] This proves that $h^{'} \circ h$ is a proximal cofibration. 
	
	 \textbf{iii)} Let $h_{j} : (X_{j},\delta_{j}) \rightarrow (X_{j}^{'},\delta_{j}^{'})$ be a family of cofibrations for all $j \in J$. Then we shall show that $\sqcup_{j} h_{j} : (\sqcup_{j} X_{j},\delta) \rightarrow (\sqcup_{j} X_{j}^{'},\delta^{'})$ is a cofibration. Since for all $j \in J$, $h_{j} : (X_{j},\delta_{j}) \rightarrow (X_{j}^{'},\delta_{j}^{'})$ is cofibration, we have that for any pc-map $k_{j}$ from $(X_{j}^{'},\delta^{'})$ to $(X^{''},\delta^{''})$ and prox-hom $F_{j} : (X_{j} \times I,\delta_{2}) \rightarrow (X^{''},\delta^{''})$ with \[k_{j} \circ (h_{j} \times 1_{0}) = F_{j} \circ i_{0}^{X_{j}},\] there is a prox-hom $F^{'}_{j} : (X_{j}^{'} \times I,\delta_{2}^{'}) \rightarrow (X^{''},\delta^{''})$ such that $F^{'}_{j} \circ i_{0}^{X_{j}^{'}} = k_{j}$ and $F^{'}_{j} \circ (h_{j} \times 1_{I}) = F_{j}$. Now assume that for a pc-map $\sqcup k_{j} : (\sqcup_{j} X_{j}^{'},\delta^{'}) \rightarrow (X^{''},\delta^{''})$ and prox-hom $\sqcup_{j} F_{j} : (\sqcup_{j} X_{j} \times I,\delta_{4}) \rightarrow (X^{''},\delta^{''})$ with \[\sqcup_{j} k_{j} \circ (\sqcup_{j} h_{j} \times 1_{0}) = \sqcup_{j} F_{j} \circ i_{0}^{\sqcup_{j} X_{j}^{'}}.\] Then there exists a map $\sqcup_{j} F^{'}_{j} : (\sqcup_{j} X_{j}^{'} \times I,\delta_{5}) \rightarrow (X^{''},\delta^{''})$ such that $F^{'}_{j} = \sqcup_{j} F^{'}_{j} \circ i_{j}$ for a map $i_{j} : (X_{j}^{'} \times I,\delta_{2}^{'}) \rightarrow (\sqcup_{j} X_{j}^{'} \times I,\delta_{5})$. If we define $i_{j}^{'}$ as $i_{j} \circ i_{0}^{X_{j}^{'}}$, then we find that $i_{0}^{\sqcup_{j} X_{j}^{'}} = \sqcup_{j} i_{j}^{'}$. It follows that
	 \begin{eqnarray*}
	 	\sqcup_{j} F^{'}_{j} \circ i_{0}^{\sqcup_{j} X_{j}^{'}} = \sqcup_{j} k_{j}
	 \end{eqnarray*} 
     and 
     \begin{eqnarray*}
     	\sqcup_{j} F^{'}_{j} \circ (\sqcup_{j} h_{j} \times 1_{I}) = \sqcup_{j} F_{j}.
     \end{eqnarray*}
     Finally, we have that $\sqcup_{j} h_{j}$ is a cofibration.
     
     \textbf{iv)} Let $h : (X,\delta) \rightarrow (X^{'},\delta^{'})$ be a proximal cofibration, i.e., there is a prox-hom $F^{'} : (X^{'} \times I,\delta_{2}^{'}) \rightarrow (X^{''},\delta^{''})$ such that \[F^{'} \circ i_{0}^{X^{'}} = k \ \ \text{and} \ \ F^{'} \circ (h \times 1_{I}) = F\] for any pc-map $k : (X^{'},\delta^{'}) \rightarrow (X^{''},\delta^{''})$ and prox-hom $F$ from $(X \times I,\delta_{2})$ to $(X^{''},\delta^{''})$ with $k \circ (h \times 1_{0}) = F \circ i_{0}^{X}$. Since we have a pushout diagram, it follows that $l^{'} \circ h = h^{'} \circ l$ holds. Now assume that $k^{'} : (Y^{'},\delta_{1}^{'}) \rightarrow (X^{''},\delta^{''})$ is a pc-map, and $F^{''}$ from $(Y^{'} \times I,\delta_{3})$ to $(X^{''},\delta^{''})$ is a prox-hom with $k^{'} \circ l^{'} = k$, $F^{''} \circ (l \times 1_{I})$, and \[k^{'} \circ (h^{'} \times 1_{0}) = F^{''} \circ i_{0}^{Y^{'}}.\] Then there exists a prox-hom \[F^{'''} : (Y^{'} \times I,\delta_{3}) \rightarrow (X^{''},\delta^{''})\] such that $F^{'''} \circ (l^{'} \times 1_{I}) = F^{'}$. Moreover, we have that
     \begin{eqnarray*}
     	F^{'''} \circ (l^{'} \times 1_{I}) = F^{'} &\Rightarrow& F^{'''} \circ (l^{'} \times 1_{I}) \circ i_{0}^{X^{'}} = F^{'} \circ i_{0}^{X^{'}} \\
     	&\Rightarrow& F^{'''} \circ i_{0}^{Y^{'}} \circ l^{'} = k \\
     	&\Rightarrow& F^{'''} \circ i_{0}^{Y^{'}} \circ l^{'} = k^{'} \circ l^{'} \\
     	&\Rightarrow& F^{'''} \circ i_{0}^{Y^{'}} = k^{'},
     \end{eqnarray*}
     and
     \begin{eqnarray*}
     	F^{'} \circ (h \times 1_{I}) = F^{'} &\Rightarrow& F^{''} \circ (l^{'} \times 1_{I}) \circ (h \times 1_{I}) = F^{''} \circ (l \times 1_{I}) \\
     	&\Rightarrow& F^{'''} \circ ((l^{'} \circ h) \times 1_{I}) = F^{''} \circ (l \times 1_{I}) \\
     	&\Rightarrow& F^{'''} \circ ((h^{'} \circ l) \times 1_{I}) = F^{''} \circ (l \times 1_{I}) \\
     	&\Rightarrow& F^{'''} \circ (h^{'} \times 1_{I}) \circ (l \times 1_{I}) = F^{''} \circ (l \times 1_{I}) \\
     	&\Rightarrow& F^{'''} \circ (h^{'} \times 1_{I}) = F^{''}.
     \end{eqnarray*}
     As a consequence, $h^{'}$ is a proximal cofibration.
\end{proof}

\begin{theorem}
	$h : (X,\delta) \rightarrow (X^{'},\delta^{'})$ is a proximal cofibration if and only if $(X^{'} \times 0) \cup (X \times I)$ is a proximal retract of $X^{'} \times I$.
\end{theorem}

\begin{proof}
	Let $X^{''} = (X^{'} \times 0) \cup (X \times I)$. If $f$ is a proximal cofibration, then for any pc-map $k : (X^{'},\delta^{'}) \rightarrow (X^{''},\delta^{''})$ and prox-hom $F : (X \times I,\delta_{2}) \rightarrow (X^{''},\delta^{''})$ with $k \circ (h \times 1_{0}) = F \circ i_{0}^{X}$, there is a prox-hom $F^{'} : (X^{'} \times I,\delta_{2}^{'}) \rightarrow (X^{''},\delta^{''})$ such that $F^{'} \circ i_{0}^{X^{'}} = k$ and $F^{'} \circ (h \times 1_{I}) = F$. Hence, $F^{'}$ is a proximal retraction of $X^{'} \times I$. Conversely, let $h : (X^{'} \times I,\delta_{2}^{'}) \rightarrow (X^{''},\delta^{''})$ be a proximal retraction. Assume that $k : (X^{'},\delta^{'}) \rightarrow (Y,\delta)$ is a pc-map and $F : (X \times I,\delta_{2}) \rightarrow (Y,\delta)$ is a prox-hom with \[k \circ (h \times 1_{0}) = F \circ i_{0}^{X}.\] Define a map $F^{''} : (X^{''},\delta^{''}) \rightarrow (Y,\delta)$ by $F^{''}(x^{'},t) = F(x^{'},t)$ and $F^{''}(x^{'},0) = k(x)$. By Lemma \ref{l1}, $F^{''}$ is a pc-map. Therefore, the map $F^{'} = F^{''} \circ k$ is a proximal fibration satisfying that $F^{'} \circ i_{0}^{X^{'}} = k$ and $F^{'} \circ (h \times 1_{I}) = F$. This shows that $h$ is a proximal cofibration.
\end{proof}

\section{Descriptive Proximity Definitions}
\label{sec:3}

\quad This section is dedicated to describing the concepts given in Section \ref{sec:2} on descriptive proximity spaces. Recall that a (spatial) proximity is also a descriptive proximity, and note that, in the examples of this section, descriptions of feature vectors consider the colors of boxes or some parts of balls (see Example \ref{ex1}, Example \ref{ex2}, and Example \ref{ex3}).

\begin{definition}
	Let $(X,\delta^{1}_{\Phi})$ and $(Y,\delta^{2}_{\Phi})$ be two descriptive proximity spaces. The descriptive proximal mapping space $Y^{X}$ is defined as the set \[\{\alpha : X \rightarrow Y \ | \ \alpha \ \text{is a dpc-map}\}\] having the following descriptive proximity relation $\delta_{\Phi}$ on itself: Let $E$, $F \subset X$ and $\{\alpha_{i}\}_{i \in I}$ and $\{\beta_{j}\}_{j \in J}$ be any subsets of dpc-maps in $Y^{X}$. We say that $\{\alpha_{i}\}_{i \in I} \delta_{\Phi} \{\beta_{j}\}_{j \in J}$ if the fact $E \delta^{1}_{\Phi} F$ implies that $\alpha_{i}(E) \delta^{2}_{\Phi} \beta_{j}(F)$.
\end{definition}

\begin{example}\label{ex1}
	Consider the set $X = \{a,b,c,d,e,f,g,h\}$ in Figure \ref{fig:1} with the descriptive proximity $\delta_{\Phi}$, where $\Phi$ is a set of probe functions that admits colors of given boxes. Define three descriptive proximal paths $\gamma_{1}$, $\gamma_{2}$, and $\gamma_{3} \in X^{I}$ by
	\begin{eqnarray*}
		&&\gamma_{1} : a \mapsto b \mapsto c \mapsto d, \\
		&& \gamma_{2} : c \mapsto b \mapsto a \mapsto h, \\
		&& \gamma_{3} : a \mapsto h \mapsto g \mapsto f.
	\end{eqnarray*}
	For all $t \in I$, $\gamma_{1}(t) \delta_{\Phi} \gamma_{2}(t)$. Indeed,
	\begin{eqnarray*}
		\gamma_{1}(t) = 
		\begin{cases}
			\text{red}, & t \in [0,1/4] \ \text{and} \ [3/4,1] \\
			\text{green}, & t \in [1/4,2/4] \\
			\text{black}, & t \in [2/4,3/4]
		\end{cases} \ \
	    = \gamma_{2}(t),
	\end{eqnarray*} 
    namely that, $\gamma_{1}$ is descriptively near $\gamma_{2}$. However, for $t \in [1/4,2/4]$, we have that $\alpha_{1}(t) =$ green and $\alpha_{3}(t) =$ black, that is, $\alpha_{1}$ and $\alpha_{3}$ are not descriptively near in $X$. 
\end{example}

\begin{definition}
	We say that a map $H : (X,\delta^{1}_{\Phi}) \rightarrow (Z^{Y},\delta_{\Phi}^{'})$ is descriptive proximally continuous if the fact $E \delta^{1}_{\Phi} F$ implies that $H(E) \delta_{\Phi}^{'} H(F)$ for any subsets $E$, $F \subset X$.
\end{definition} 

\begin{definition}
	For any descriptive proximity spaces $(X,\delta^{1}_{\Phi})$ and $(Y,\delta^{2}_{\Phi})$, the descriptive proximal evaluation map \[e_{X,Y} : (Y^{X} \times X,\delta_{\Phi}) \rightarrow (Y,\delta^{2}_{\Phi})\] is defined by $e(\alpha,x) = \alpha(x)$.
\end{definition}  

\begin{proposition}
	The descriptive proximal evaluation map $e_{X,Y}$ is a dpc-map.
\end{proposition}

\begin{proof}
	We shall show that for any $E$, $F \subset X$ and $\{\alpha_{i}\}_{i \in I}$, $\{\beta_{j}\}_{j \in J} \subset Y^{X}$, $(\{\alpha_{i}\}_{i \in I} \times E) \delta_{\Phi} (\{\beta_{j}\}_{j \in J} \times F)$ implies $e_{X,Y}(\{\alpha_{i}\}_{i \in I} \times E) \delta^{2}_{\Phi} e_{X,Y}(\{\beta_{j}\}_{j \in J} \times F)$. 
	\begin{eqnarray*}
		(\{\alpha_{i}\}_{i \in I} \times E) \delta_{\Phi} (\{\beta_{j}\}_{j \in J} \times F) \ \ &\Rightarrow& \{\alpha_{i}\}_{i \in I} \delta_{\Phi}^{'} \{\beta_{j}\}_{j \in J} \ \ \text{and} \ \ E \delta^{1}_{\Phi} F \\
		&\Rightarrow& \alpha_{i}(E) \delta^{2}_{\Phi} \beta_{j}(F), \ \ \forall i \in I, \ \forall j \in J \\
		&\Rightarrow& e_{X,Y}(\{\alpha_{i}\}_{i \in I} \times E) \delta^{2}_{\Phi} e_{X,Y}(\{\beta_{j}\}_{j \in J} \times F),
	\end{eqnarray*}
    where $Y^{X}$ has a descriptive proximity $\delta_{\Phi}^{'}$.
\end{proof}

\begin{definition}\label{d3}
	A surjective and dpc-map $p : (X,\delta_{\Phi}) \rightarrow (X^{'},\delta_{\Phi}^{'})$ between any descriptive proximity spaces $(X,\delta_{\Phi})$ and $(X^{'},\delta_{\Phi}^{'})$ is a descriptive proximal covering map if the following hold:
	\begin{itemize}
		\item Let $\{x^{'}\} \subseteq X^{'}$ be any subset with $\{x^{'}\} \ll_{\delta_{\Phi}^{'}} Y^{'}$. Then there is an index set $I$ satisfying that
		\begin{eqnarray*}
			p^{-1}(Y^{'}) = \displaystyle\bigcup_{i \in I}Y_{i}
		\end{eqnarray*}
		with $V_{i} \ll_{\delta_{\Phi}} Y_{i}$, where $V_{i} \in p^{-1}(\{x^{'}\})$ for each $i \in I$.
		\item $Y_{i} \neq Y_{j}$ when $i \neq j$ for $i$, $j \in I$.
		\item $p|_{Y_{i}} : Y_{i} \rightarrow Y^{'}$ is a descriptive proximal isomorphism for every $i \in I$.  
	\end{itemize}
\end{definition}

\quad In Definition \ref{d3}, $(X,\delta_{\Phi})$ is called a descriptive proximal covering space of $(X^{'},\delta_{\Phi}^{'})$. For $i \in I$, $Y_{i}$ is said to be a descriptive proximal sheet. For any $x^{'} \in X^{'}$, $p^{-1}(\{x^{'}\})$ is called a descriptive proximal fiber of $x^{'}$. The map $p|_{Y_{i}} : Y_{i} \rightarrow Y^{'}$ is a descriptive proximal isomorphism if the map $p : (X,\delta_{\Phi}) \rightarrow (X^{'},\delta_{\Phi}^{'})$ is a descriptive proximal isomorphism. However, the converse is not generally true. Given any descriptive proximity space $(X,\delta_{\Phi})$, it is obvious that the identity map on $X$ is always a descriptive proximal covering map.  

\begin{example}\label{ex2}
	Consider the surjective and dpc-map $p : (X,\delta_{\Phi}) \rightarrow (X^{'},\delta_{\Phi}^{'})$, defined by $p(a_{i}) = p(b_{i}) = p(c_{i}) = d_{i}$ for any $i = 1,2,3,4$, in Figure \ref{fig:2}, where $\Phi$ is a set of probe functions that admits colors of given shapes. Let $\{d_{1}\} \subset X^{'}$ and $Y^{'} = \{d_{1},d_{3},d_{4}\}$ a $\delta_{\Phi}^{'}-$neighborhood of $\{d_{1}\}$. For $V_{1} = \{a_{1}\}$, $V_{2} = \{b_{1}\}$, and $V_{3} = \{c_{1}\}$, we have that $p^{-1}(Y^{'}) = \displaystyle \bigcup_{i=1}^{3}Y_{i}$, where $Y_{1} = \{a_{1},a_{3},a_{4}\}$, $Y_{2} = \{b_{1},b_{3},b_{4}\}$, and $Y_{3} = \{c_{1},c_{3},c_{4}\}$. This gives us that for all $i \in \{1,2,3\}$, $Y_{i}$ is a $\delta_{\Phi}-$neighborhood of $V_{i}$. We also observe that $Y_{i} \neq Y_{j}$ if $i \neq j$ for $i$, $j \in \{1,2,3\}$. In addition, $p|_{Y_{i}} : Y_{i} \rightarrow Y^{'}$ is a descriptive proximal isomorphism for each $i$. If one considers $d_{3}$ and $d_{4}$, the same process can be repeated. Let $\{d_{2}\} \ll_{\delta_{\Phi}^{'}} \{d_{2}\} = Y^{'}$ in $X^{'}$. Then $p^{-1}(Y^{'}) = Y_{1} \cup Y_{2} \cup Y_{3}$, where $Y_{1} = \{a_{2}\}$, $Y_{2} = \{b_{2}\}$, and $Y_{3} = \{c_{2}\}$. We observe that $V_{1} = \{a_{2}\} \ll_{\delta_{\Phi}} Y_{1}$, $V_{2} = \{b_{2}\} \ll_{\delta_{\Phi}} Y_{2}$, and $V_{3} = \{c_{2}\} \ll_{\delta_{\Phi}} Y_{3}$. Note that $Y_{1} \neq Y_{2} \neq Y_{3}$. Furthermore, $p|_{Y_{i}} : Y_{i} \rightarrow Y^{'}$ is a descriptive proximal isomorphism for each $i = 1,2,3$. This proves that $p$ is a descriptive proximal covering map.
\end{example}

\begin{definition}
	A dpc-map $p : (X,\delta_{\Phi}) \rightarrow (X^{'},\delta_{\Phi}^{'})$ is said to have the descriptive proximal homotopy lifting property (DPHLP) with respect to a descriptive proximity space $(X^{''},\delta_{\Phi}^{''})$ if, for an inclusion map $i_{0} : (X^{''},\delta_{\Phi}^{''}) \rightarrow (X^{''} \times I,\delta^{1}_{\Phi})$ defined by $i_{0}(x^{''}) = (x^{''},0)$, for every dpc-map $h : (X^{''},\delta_{\Phi}^{''}) \rightarrow (X,\delta_{\Phi})$, and dprox-hom $G : (X^{''} \times I,\delta^{1}_{\Phi}) \rightarrow (X^{'},\delta_{\Phi}^{'})$ with $p \circ h = G \circ i_{0}$, then there exists a dprox-hom $G^{'} : (X^{''} \times I,\delta^{1}_{\Phi}) \rightarrow (X,\delta_{\Phi})$ for which $G^{'}(x^{''},0) = h(x^{''})$ and $p \circ G^{'}(x^{''},t) = G(x^{''},t)$.
	\begin{displaymath}
		\xymatrix{
			X^{''} \ar[r]^{h} \ar[d]_{i_{0}} &
			X \ar[d]^{p} \\
			X'' \times I \ar[r]_{G} \ar@{.>}[ur]^{G^{'}} & X^{'}. }
	\end{displaymath}
\end{definition}

\begin{definition}
	A map $p : (X,\delta_{\Phi}) \rightarrow (X^{'},\delta_{\Phi}^{'})$, which is a dpc-map, is said to be a descriptive proximal fibration if it has the DPHLP for any descriptive proximity space $(X^{''},\delta_{\Phi}^{''})$.
\end{definition}

\begin{definition}\label{d6}
	Given two descriptive proximity spaces $(X,\delta_{\Phi})$ and $(X^{'},\delta_{\Phi}^{'})$, a dpc-map $h : X \rightarrow X^{'}$ is said to have a dprox-hom extension property (DPHEP) with respect to a descriptive proximity space $(X^{''},\delta_{\Phi}^{''})$ if there exists a dprox-hom \[F^{'} : (X^{'} \times I,\delta_{\Phi}^{1'}) \rightarrow (X^{''},\delta_{\Phi}^{''})\] satisfying the conditions $F^{'} \circ i_{0}^{X^{'}} = k$ and $F^{'} \circ (h \times 1_{I}) = F$ for any dpc-map $k : (X^{'},\delta_{\Phi}^{'}) \rightarrow (X^{''},\delta_{\Phi}^{''})$, and dprox-hom $F : (X \times I,\delta_{\Phi}^{1}) \rightarrow (X^{''},\delta_{\Phi}^{''})$ with the equality $k \circ (h \times 1_{0}) = F \circ i_{0}^{X}$, where the maps $i_{0}^{X} : (X,\delta_{\Phi}) \rightarrow (X \times I,\delta^{1}_{\Phi})$ and $i_{0}^{X^{'}} : (X^{'},\delta_{\Phi}^{'}) \rightarrow (X^{'} \times I,\delta_{\Phi}^{1'})$ are inclusions.
\end{definition}
\begin{displaymath}
	\xymatrix{X \times 0 \ar@{^{(}->}[rr]^{i_{0}^{X}} \ar[dd]_{h \times 1_{0}} & & X \times I \ar[dd]^{h \times 1_{I}} \ar[dl]_{F} \\
		& X^{''} &\\
		X^{'} \times 0 \ar@{^{(}->}[rr]_{i_{0}^{X^{'}}} \ar[ur]^{k} & & X^{'} \times I. \ar@{.>}[ul]_{F^{'}}}
\end{displaymath}

\begin{definition}
	A dpc-map $f : (X,\delta_{\Phi}) \rightarrow (X^{'},\delta_{\Phi}^{'})$ is said to be a descriptive proximal cofibration if it has the DPHEP with respect to any descriptive proximity space $(X^{''},\delta_{\Phi}^{''})$.
\end{definition}

\begin{example}\label{ex3}
	Let $(X^{''},\delta_{\Phi}^{''})$ be a descriptive proximity space as in Figure \ref{fig:3}, where $\Phi$ is a set of probe functions which admits colors of given rounds. Assume that $\gamma_{1}$ and $\gamma_{2}$ are descriptive proximal paths on $X^{''}$ such that $\gamma_{1}$ is a descriptive proximal path from $b$ to $a$ and $\gamma_{2}$ is a descriptive proximal path from $b$ to $c$. Let $h : (\{0\},\delta_{\Phi}) \rightarrow (I,\delta_{\Phi}^{'})$ be an inclusion map. For a dpc-map $k : (I,\delta_{\Phi}^{'}) \rightarrow (X^{''},\delta_{\Phi}^{''})$ defined as $k = \gamma_{2}$, and a dprox-hom $F : (\{0\} \times I,\delta_{\Phi}^{1}) \rightarrow (X^{''},\delta_{\Phi}^{''})$ defined by $F(0,t) = \alpha(t)$ for all $t \in I$ with the property $k \circ (h \times 1_{0}) = F \times i_{0}^{\{0\}}$, there exists a dprox-hom \[F^{'} : (I \times I,\delta_{\Phi}^{1'}) \rightarrow (X^{''},\delta_{\Phi}^{''})\] defined by $F^{'}(0,t_{1}) = F(0,t_{1})$ and $F^{'}(t_{2},0) = k(t_{2})$ for all $(t_{1},t_{2}) \in I \times I$ which satisfy
	\begin{eqnarray*}
		&&F^{'} \circ (h \times 1_{I}) = F,\\
		&&F^{'} \circ i_{0}^{I} = k.
	\end{eqnarray*}
    In another saying, the diagram
    \begin{displaymath}
    	\xymatrix{\{0\} \times 0 \ar@{^{(}->}[rr]^{i_{0}^{\{0\}}} \ar@{^{(}->}[dd]_{h \times 1_{0}} & & \{0\} \times I \ar@{^{(}->}[dd]^{h \times 1_{I}} \ar[dl]_{F} \\
    		& X^{''} &\\
    		I \times 0 \ar@{^{(}->}[rr]_{i_{0}^{I}} \ar[ur]^{k} & & I \times I \ar@{.>}[ul]_{F^{'}}}
    \end{displaymath}
    holds.
\end{example}

\section{Conclusion}
\label{sec:4}
\quad A subfield of topology called homotopy theory investigates spaces up to continuous deformation.  Although homotopy theory began as a topic in algebraic topology, it is currently studied as an independent discipline. For instance, algebraic and differential nonlinear equations emerging in many engineering and scientific applications can be solved using homotopy approaches. As an example, these equations include a set of nonlinear algebraic equations that model an electrical circuit. In certain studies, the aging process of the human body is presented using the algebraic topology notion of homotopy. In addition to these examples, one can easily observe homotopy theory once more when considering the algorithmic problem of robot motion planning. In this sense, this research is planned to accelerate homotopy theory studies within proximity spaces that touch many important application areas. Moreover, this examination encourages not only homotopy theory but also homology and cohomology theory to take place within proximity spaces. The powerful concepts of algebraic topology always enrich the proximity spaces and thus it becomes possible to see the topology even at the highest level fields of science such as artificial intelligence and medicine.

\acknowledgment{The second author is grateful to the Azerbaijan State Agrarian University for all their hospitality and generosity during his stay. This work has been supported by the Scientific and Technological Research Council of Turkey TÜBİTAK-1002-A with project number 122F454.}


\begin{thebibliography}{99}

\bibitem{Efremovic:1952} V.A. Efremovic, The geometry of proximity I, Matematicheskii Sbornik(New Series), 31(73), 189-200, (1952).

\bibitem{Kuratowski:1958} C. Kuratowski, Topologie. I, Panstwowe Wydawnictwo Naukowe, Warsaw, xiii+494 pp., (1958).

\bibitem{Leader:1964} S. Leader, On products of proximity spaces, Mathemathische Annalen, 154, 185-194, (1964).

\bibitem{Lodato:1964} M.W. Lodato, On topologically induced generalized proximity relations, Proceedings of the American Mathematical Society, 15, 417-422, (1964).

\bibitem{NaimpallyWarrack:1970} S.A. Naimpally, and B.D. Warrack, Proximity Spaces, Cambridge Tract in Mathematics No. 59, Cambridge University Press, Cambridge, UK, x+128 pp., Paperback (2008), MR0278261 (1970).

\bibitem{NaimpallyPeters:2013} S.A. Naimpally, and J.F. Peters, Topology With Applications. Topological
Spaces via Near and Far, World Scientific, Singapore (2013).

\bibitem{MrowkaPervin:1964} S.G. Mrowka, and W.J. Pervin, On uniform connectedness, Proceedings of the American Mathematical Society, 15(3), 446-449, (1964).

\bibitem{PeiRen:1985} F. Pei-Ren, Proximity on function spaces, Tsukuba Journal of Mathematics, 9(2), 289-297, (1985). 

\bibitem{Peters1:2007} J.F. Peters, Near sets. General theory about nearness of objects, Applied
Mathematical Sciences, 1(53), 2609-2629, (2007). 

\bibitem{Peters2:2007} J.F. Peters, Near sets. Special theory about nearness of objects, Fundamenta Informaticae, 75(1-4), 407-433, (2007).

\bibitem{PetersNaimpally:2012} J.F. Peters, and S.A. Naimpally, Applications of near sets, American Mathematical Society Notices, 59(4), 536-542, (2012).

\bibitem{Peters:2013} J.F. Peters, Near sets: An introduction, Mathematics in Computer Science, 7(1), 3-9, (2013).

\bibitem{Peters:2014} J.F. Peters, Topology of Digital Images: Visual Pattern Discovery in Proximity Spaces (Vol. 63), Springer Science \& Business Media (2014).

\bibitem{PetersTane:2021} J.F. Peters, and T. Vergili, Good Coverings of Proximal Alexandrov Spaces. Homotopic Cycles in Jordan Curve Theorem Extension., arXiv preprint arXiv:2108.10113 (2021).

\bibitem{PetersTane2:2021} J.F. Peters, and T. Vergili, Descriptive Proximal Homotopy. Properties and Relations., arXiv preprint arXiv:2104.05601v1 (2021).

\bibitem{PetersTane:2022} J.F. Peters, and T. Vergili, Good Coverings of Proximal Alexandrov Spaces. Path Cycles in the Extension of the Mitsuishi-Yamaguchi good covering and Jordan Curve Theorems., Applied General Topology, 24(1), 25-45, (2023).

\bibitem{PetersTane2:2022} J.F. Peters, and T. Vergili, Proximity Space Categories. Results For Proximal Lyusternik-Schnirel'man, Csaszar And Bornology Categories., submitted to Afrika Matematika (2022).

\bibitem{Poincare:1895} H. Poincare, Analysis Situs, Paris, France: Gauthier-Villars (1895).

\bibitem{Riemann:1851} B. Riemann, Grundlagen für eine allgemeine Theorie der Functionen einer veränderlichen complexen Grösse, Huth, (1851).

\bibitem{Riesz:1908} F. Riesz, Stetigkeitsbegriff und abstrakte mengenlehre, Atti del IV Congresso Internazionale dei Matematici II, 18-24pp, (1908).

\bibitem{Smirnov:1952} Y.M. Smirnov, On proximity spaces, Matematicheskii Sbornik(New Series), 31(73), 543-574, (1952). English Translation: American Mathematical Society Translations: Series 2, 38, 5-35, (1964).

\bibitem{Wallace:1941} A.D. Wallace, Separation Space, Annals of Mathematics, 687-697, (1941). 

\end{thebibliography}
\end{document}